\documentclass{amsart}
\newtheorem{Theorem}{Theorem}
\newtheorem{Proposition}{Proposition}
\newtheorem{Lemma}{Lemma}
\newtheorem{Corollary}{Corollary}
\newcommand{\Ecm}{\Gamma}
\newcommand{\Eam}{\Omega}
\newcommand{\rd}{{\mathbb R}^d}
\newcommand{\R}{{\mathbb R}}

\newcommand{\BK}{{\mathbb K}}

\newcommand{\cal}{\mathcal}
\newcommand{\ep}{\varepsilon}
\newcommand{\Lin}{{\rm Lin}}

\newcommand{\nor}{{\rm nor}}
\newcommand{\card}{{\rm card}}

\newcommand{\Forth}{F^{\perp}}
\newcommand{\bark}{k^*}
\newcommand{\barl}{l^*}
\title[Integral representation of mixed volumes]{A product 
integral representation of mixed volumes of two convex bodies}

\author{Daniel Hug}
\thanks{The authors are grateful for support from the German Science Foundation (DFG) and 
the Czech Science Foundation (GA\v{C}R 201/10/J039) for the joint project ``Curvature Measures and Integral Geometry''} 
\address{Karls\-ruhe Institute of Technology, Department of Mathematics, 
D-76128 Karls\-ruhe, Germany}
\email{daniel.hug@kit.edu}
\urladdr{http://www.math.kit.edu/$\sim$hug/}

\author{Jan Rataj}
\address{Charles University, Faculty of Mathematics and Physics, 
Sokolovska 83, 186 75 Praha 8, 
Czech Republic}
\email{rataj@karlin.mff.cuni.cz}
\urladdr{http://www.karlin.mff.cuni.cz/$\sim$rataj/index\underline{ }en.html}

\author{Wolfgang Weil}
\address{Karls\-ruhe Institute of Technology, Department of Mathematics, 
D-76128 Karls\-ruhe, Germany}
\email{wolfgang.weil@kit.edu}
\urladdr{http://www.math.kit.edu/$\sim$weil/}

\date{\today}
\thanks{The authors are grateful to a referee for useful remarks which helped to improve the presentation of the paper.}
\begin{document}

\begin{abstract}
The Brunn-Minkowski theory in convex geometry relies heavily on the notion of mixed volumes.  
Despite its particular importance, even explicit representations for the mixed volumes 
of two convex bodies in Euclidean space $\rd$ are available only in special cases. 
Here we investigate a new integral representation of such mixed volumes, in terms of 
flag measures of the involved convex sets. A brief introduction to (extended) flag measures 
of convex bodies is also provided.  
\end{abstract}

\maketitle

\section*{Authors' Addresses}

\noindent
Daniel Hug, Karlsruhe Institute of Technology, Department of Mathematics, 
D-76128 Karlsruhe, Germany, 
\email{daniel.hug@kit.edu}

\medskip

\noindent
Jan Rataj, Charles University, Faculty of Mathematics and Physics, 
Sokolovska 83, 186 75 Praha 8, 
Czech Republic, 
\email{rataj@karlin.mff.cuni.cz}

\medskip

\noindent
Wolfgang Weil, Karlsruhe Institute of Technology, Department of Mathematics, 
D-76128 Karlsruhe, Germany, 
\email{wolfgang.weil@kit.edu}

\section*{Running Title}

\noindent
Integral representation of mixed volumes

\section{Introduction}

Mixed volumes of convex bodies are a fundamental concept and tool 
in the classical Brunn-Minkowski theory of convex geometry. For two 
convex bodies (non-empty compact convex sets) in $\R^d, d\ge 2$, the mixed volumes 
$$
V(K [m], M [d-m]),\quad m=1,\dots ,d-1,$$
appear as coefficients in the generalized Steiner formula
$$
V_d(K+M) = \sum_{j=0}^d \binom{d}{ j}V(K [j], M [d-j]) 
$$
for the volume of the Minkowski sum $K+M$ of $K$ and $M$. Other classical notions, the support function $h(K,\cdot)$ of $K$ and the  area measure $S_{d-1}(M,\cdot )$ of $M$, are related to special mixed volumes through an integration over the unit sphere $S^{d-1}$,
$$
V(K [1], M [d-1]) = \frac{1}{d}\int_{S^{d-1}} h(K,u)\,S_{d-1}(M,{ d}u),
$$
a result which holds for all convex bodies $K,M$ (see \cite{S}, for details on the Brunn-Minkowski theory). Under some smoothness and symmetry assumptions, a similarly simple decomposition  exists 
for the  mixed volumes $V(K [m], M [d-m])$ with  $m\in\{2,...,d-2\}$. Namely, for $m\in \{0,\dots,d-1\}$,  
\begin{equation}\label{mv2}
V(K [m], M [d-m]) = \frac{2^{d-m}m!}{d!}\int_{G(d,d-m)} V_m(K|E^\perp)\,\rho_{d-m}(M,{d}E)
\end{equation} 
holds for instance if $M$ is a centrally symmetric and smooth 
body, whereas $K$ may be arbitrary. Here, $G(d,d-m)$ is the Grassmannian of $(d-m)$-dimensional 
linear subspaces of $\R^d$, $V_m(K|E^\perp)$ is the $m$-dimensional volume of the 
orthogonal projection of $K$ onto the orthogonal complement $E^\perp\in G(d,m)$ of $E\in G(d,d-m)$,  
 and the signed measure $\rho_{d-m}(M,\cdot )$ is the $(d-m)$th projection 
 generating measure of $M$, normalized as in \cite[p.~1315]{GW93}. If $K$ is  
 centrally symmetric and smooth, then the projection function $V_m(K|\cdot)$ has the integral representation
 \begin{equation}\label{mv2a}
V_m(K|E^\perp) = \frac{2^m}{m!}\int_{G(d,m)}|\langle E^\perp,F\rangle|\,
 \rho_m(K,{d}F),
\end{equation} 
where $|\langle E^\perp, F\rangle|$ denotes the absolute value of the determinant of the 
orthogonal projection of $E^\perp$ onto $F$. Hence in this case, where both bodies $K$ and $M$ are centrally symmetric and smooth, we obtain the  relation
\begin{eqnarray}\label{mv3}
&&V(K [m], M [d-m]) \nonumber \\
&&\qquad\qquad \, = \frac{2^d}{d!}\int_{G(d,d-m)} \int_{G(d,m)}|\langle E^\perp,F\rangle|\,
 \rho_m(K,{d}F)\,\rho_{d-m}(M,{d}E).
\end{eqnarray} 
Let  $\mathbb{F}_{m,d-m}(K,M)$ denote the right-hand side of this equation. It involves the $m$th and 
the $(d-m)$th projection generating measure,     
$\rho_m(K,\cdot)$ and $\rho_{d-m}(M,\cdot)$, of $K$ and $M$, respectively, as well as basic information about 
the relative position of the subspaces $E$ and $F$. Since $|\langle E^\perp,F\rangle|=|\langle F^\perp,E\rangle|$,  
we have the symmetry relation $ \mathbb{F}_{m,d-m}(K,M)=\mathbb{F}_{d-m,m}(M,K)$.

It is known that \eqref{mv2} holds for arbitrary $K$ and generalized zonoids $M$, \eqref{mv2a} holds for generalized zonoids $K$ and, therefore, \eqref{mv3} remains true if $K$ and $M$ are both generalized zonoids. For an introduction to zonoids and generalized zonoids, we 
refer to the surveys \cite{SWsurvey83,GW93} and to \cite{S}.  
 Relations \eqref{mv2a} and \eqref{mv3} easily follow,  for instance, from  \cite[Theorem 2.5]{GW93} (see also \cite[Theorem 5.3.1]{S}). The more general relation \eqref{mv2}, which also 
yields \eqref{mv2a} and \eqref{mv3} as simple consequences, can be deduced from the multilinearity 
of mixed volumes by first considering the mixed volume $V_d(K[m],M_1,\ldots,M_{d-m})$ 
with (generalized) zonoids $M_1,\ldots,M_{d-m}$ (cf.\ \cite[(9.7)]{SWsurvey83}).  In spite of the generality 
of relation \eqref{mv2}, it does not even seem possible to obtain a similar result (for $m>1$) for all 
centrally symmetric  polytopes as long as such a relation is based 
on integrals over Grassmannians. 
In fact, if \eqref{mv2} holds for a symmetric polytope $M$ (with interior points) and all smooth symmetric bodies $K$, then $M$ lies in the class $\mathcal{Q}_s(d-m,m)$, considered in \cite{GW91}. But then 
Theorem 4.1 in \cite{GW91} together with the remark in \cite[p.~128, l.~1]{GW91}
implies that all $(d-m+1)$-dimensional faces of $M$ have to be centrally symmetric. Thus, $M$ cannot be an octahedron, for example. A similar argument shows that also \eqref{mv3} cannot be extended to all symmetric polytopes $K$ or $M$.

 In the following, we use flag measures of convex bodies to show that a formula generalizing \eqref{mv3} holds with Grassmannians replaced by certain flag manifolds, associated with the given convex bodies $K$ and $M$, respectively. The result, which we shall prove, yields
\begin{eqnarray} \label{mv4}
&&V(K[m],M[d-m])\nonumber\\
&&\qquad\qquad \,=\iint f_{m,d-m}(u,U,v,V)\, \Eam_m(K;d(u,U))\,\Eam_{d-m}(M;d(v,V)) ,
\end{eqnarray}
where $\Eam_m(K;\cdot)$ and $\Eam_{d-m}(M;\cdot)$ are flag measures of $K$ and $M$, the function $f_{m,d-m}$ is independent of $K$ and $M$, and the integration is over the manifold of flags $(u,U)$  (respectively $(v,V)$). 
 For this formula, no symmetry or smoothness assumptions on $K$ or $M$ have to be imposed. However, we have to assume 
 that $K$ and $M$ are in general relative position with respect to each other. If $K$ and $M$ are 
 polytopes, this condition is, for instance, satisfied if $K$ and $M$ do not have parallel faces of complementary 
 dimension.  See Section \ref{results}, for precise definitions and Theorem \ref{thm3} for the explicit result. 
 
 Flag measures were already used in \cite{H02} (see also \cite{GHHRW}) to provide an integral representation of projection functions. Here  we proceed in a different, more direct way and establish a representation result for special mixed volumes. Our approach is based on general integral geometric  results from \cite{RoZ92,R02} for sets of positive reach, which are applied to convex sets.  This yields extended flag  measures which are related to the measures introduced in \cite{H02}, \cite{HHW} by 
means of a local Steiner formula. The formula we thus obtain includes also a formula for projection functions, although in a less explicit form than in \cite{H02}, \cite{GHHRW}.

The setup of the paper is as follows. After some preliminaries in the next section, we  introduce, in Section 3, the extended flag measures of a convex body. In Section \ref{results}  we formulate our main results, Theorems \ref{main} and \ref{thm3}. Theorem \ref{main} provides an integral representation not for the mixed volumes but for certain $\epsilon$-approximations  of these. Theorem \ref{thm3}, which implies the representation \eqref{mv4}, is deduced from Theorem \ref{main} by an approximation argument. After some preparations in Section 5, we give the corresponding proofs in Sections 6 and 7.
 The final section contains an example which shows that a simple extension of \eqref{mv4} to bodies which are not in general relative position is not possible in general.

\section{Preliminaries}
Let $\rd$ be the Euclidean space with scalar product $\langle \cdot,\cdot\rangle$ and norm $\|\cdot\|$. The unit ball 
and the unit sphere of $\rd$ are denoted by $B^d$ and $S^{d-1}$, respectively. 
For a given $k\in\{0,\ldots,d\}$, we denote by  $\bigwedge_k\rd$ 
the $\binom{d}{k}$-dimensional linear space of $k$-vectors in $\rd$. As usual, we identify $\bigwedge_0\rd$ with $\R$. 
The vector space  $\bigwedge_k\rd$ is equipped with the scalar product $\langle\cdot ,\cdot\rangle$ (cf.~\cite[\S 1.7.5]{Federer69}).  
For simple $k$-vectors, the scalar product is given by
$$\langle u_1\wedge\cdots\wedge u_k,v_1\wedge\cdots\wedge v_k\rangle =
\det\left(\left(\langle u_i, v_j\rangle \right)_{i,j=1}^k\right),$$
where $u_1,\ldots,u_k,v_1,\ldots,v_k\in\R^d$. 
The induced norm on $\bigwedge_k\rd$ is denoted by $\|\cdot\|$.  This notation is consistent with the one for elements of $\R^d$ which can also be viewed as $1$-vectors. 
The operation of the orthogonal group $O(d)$ on $\rd$ is extended to an operation of  $O(d)$ on $\bigwedge_k\rd$ in the canonical way (\cite[\S 1.3.1]{Federer69}); see \cite[Chapter 1]{Federer69} for a brief introduction to multilinear algebra as used here. 
Let $G_0(d,k)$ be the subset of $\bigwedge_k\rd$ which consists of the simple $k$-vectors  with norm one (oriented Grassmann  manifold). 
The Grassmann manifold $G(d,k)$ of $k$-dimensional linear subspaces of $\R^d$ is the quotient space of 
  $G_0(d,k)$ with respect to the equivalence relation 
$\sim$ defined by $\xi\sim\zeta$ if and only if $\xi=\pm \zeta$, for $\xi, \zeta\in G_0(d,k)$. With a simple unit $k$-vector 
$u_1\wedge\ldots\wedge u_k$ we associate the  linear subspace $U=\{x\in\R^d:x\wedge u_1\wedge\ldots\wedge u_k=0\}$ 
which is just the $k$-dimensional linear subspace spanned by $u_1,\ldots,u_k$. Conversely, for a linear subspace $U$ we may choose an orthonormal basis $u_1,\ldots,u_k$ of $U$. Then $u_1\wedge\ldots\wedge u_k$ 
is a simple unit $k$-vector for which $U$ is the associated subspace. Moreover, up to the sign, this simple unit $k$-vector 
is uniquely determined in this way (of course, the explicit representation of the $k$-vector is not unique). See \S 1.6.1, \S  1.6.2 and, in particular, p.~267 in \cite{Federer69}, for further details and \cite{JensenKieu92} for a similar description.

The $j$-dimensional Hausdorff measure in a metric space 
will be denoted by $\mathcal{H}^j$, where we adopt the same normalization as in \cite[\S 2.10.2, p.~171]{Federer69}. 
Let $\nu_k^d$ denote the $O(d)$ invariant measure on $G(d,k)$ normalized to a probability measure. 
Thus, $\nu_k^d$ is equal to a multiple of the $k(d-k)$-dimensional Hausdorff measure on $G(d,k)$,
$$\nu_k^d=\beta(d,k)^{-1} \,{\cal H}^{k(d-k)}\llcorner G(d,k) ,$$
where $\llcorner$ denotes the restriction of a measure to a subset. 
The explicit value of the numerical constant $\beta(d,k)$ is the total Hausdorff measure of the Grassmannian which is provided in 
\cite[p.\ 267]{Federer69} and is equal to
$$
\beta(d,k)=\Gamma\left(\frac{1}{2}\right)^{k(d-k)}\prod_{j=1}^k\frac{\Gamma(\frac{j}{2})}{\Gamma(\frac{d-j+1}{2})}.
$$
The corresponding invariant probability measure on $G_0(d,k)$ is denoted 
by $\bar{\nu}^d_k$.

In the following, we consider the {\it flag manifold}
$$\Forth(d,k)=\{ (u,V)\in S^{d-1}\times G(d,k):\, u\perp V\} ,$$
where $u\perp V$ means that $u$ is orthogonal to the linear subspace $V$.  

If $K$ is a convex body in $\rd$, let $\partial K$ denote its topological boundary, and let
$$\nor (K)=\{ (x,u)\in\partial K\times S^{d-1}:\,\langle u  , y-x\rangle\leq 0\mbox{ for all } y\in K\}$$
be its unit normal bundle. This is a $(d-1)$-rectifiable set. 

Subsequently, it is convenient to use the shorthand notation 
$\bark$ for $d-1-k$, where $k\in\{0,\ldots,d-1\}$. 
The $k$th {\it support measure} $\Xi_k(K;\cdot)$ of $K$ is a measure on $\R^d\times S^{d-1}$ 
which is concentrated on $\nor( K)$ and can be represented in the form
\begin{align*}
&\int g(x,u)\,\Xi_k(K;d(x,u))\\
&\qquad =\frac 1{{\cal H}^{\bark}(S^{\bark})}\int_{\nor (K)}g(x,u)\sum_{|I|=
\bark}
\BK_I(K;x,u)\,{\cal H}^{d-1}(d(x,u)),
\end{align*}
where $g$ is any bounded measurable function on $\rd\times\ S^{d-1}$, $I$ denotes a subset of $\{ 1,\ldots ,d-1\}$ of cardinality $|I|$,
$$\BK_I(K;x,u)=\frac{\prod_{i\in I}k_i(K;x,u)}{\prod_{i=1}^{d-1}\sqrt{1+k_i(K;x,u)^2}},$$
and the numbers $k_i(K;x,u)\in [0,\infty]$ are the generalized principal curvatures of $K$ at $(x,u)\in\nor (K)$, $i=1,\ldots ,d-1$. If $k_i(K;x,u)=\infty$ for some $i\in\{1,\ldots,d-1\}$, then $\BK_I(K;x,u)$ is determined as the limit which is obtained as 
$k_i(K;x,u)\to \infty$. In particular, this implies $\frac 1{\sqrt{1+\infty^2}}=0$ and $\frac \infty{\sqrt{1+\infty^2}}=1$. Moreover, a  product over an empty index set is considered as a factor one. The generalized principal curvatures are defined for ${\cal H}^{d-1}$-almost all $(x,u)\in \nor (K)$. In the following, we do not repeat this fact (also in similar situations). We refer to \cite{Zaehle86,Hug98} for background information and an introduction to these generalized curvatures and measures from the viewpoint of geometric measure theory.
We also use the notation
$$A_I(K;x,u)=\Lin\{ a_i(K;x,u):\, i\in I\},$$
where $a_i(K;x,u)\in S^{d-1}$, $i=1,\ldots ,d-1$, are generalized principal directions of curvature of $K$ at $(x,u)$, which form an orthonormal basis of $u^\perp$ (the subspace orthogonal to $u$), and $\Lin$ denotes the linear hull. If $I=\emptyset$, then $A_I(K;x,u)=\{0\}$.   Sometimes it is convenient to consider $A_I(K;x,u)$ as a multivector  (cf.\ Section 4), i.e.\ 
$$
A_I(K;x,u)=\textstyle{\bigwedge}_{i\in I}a_i(K;x,u).
$$ 
Here, the right-hand side is $1\in\bigwedge_0\R^d$ if $I=\emptyset$. 

The support measures naturally arise as coefficients in a local Steiner formula. For a Borel   set $\eta \subset \rd\times S^{d-1}$ and  $\varepsilon\ge 0$, we define the local parallel set 
$$M_\varepsilon(K,\eta):=\{x+tu:(x,u)\in \nor (K)\cap \eta,t\in (0,\varepsilon]\}.
$$ 
Then the local Steiner formula can be expressed in the form 
$$
\mathcal{H}^d(M_\varepsilon(K,\eta))=\sum_{j=0}^{d-1}\varepsilon^{d-j}\kappa_{d-j}\,\Xi_j(K;\eta),
$$
where $\kappa_j=\mathcal{H}^j(B^j)$ is the $j$-dimensional volume of the $j$-dimensional unit ball; 
see, e.g., \cite{S,SW}. The image of $\Xi_k(K;\cdot)$ under the projection $(x,u)\mapsto u$ is the $k$th area measure $\Psi_k(K,\cdot)$ of $K$, the total measure $V_k(K)=\Xi_k(K;\rd\times S^{d-1})$ is the $k$th intrinsic volume of $K$. 
Sometimes other normalizations are used in the literature. For instance, in convex geometry
$$
S_k(K,\cdot)=\frac{d\kappa_{d-k}}{\binom{d}{k}}\,\Psi_k(K,\cdot)
$$
is often called the $k$th area measure of $K$, and we shall prefer this normalization and terminology.

\section{Flag measures}
In this section, we provide a brief introduction to flag measures as is appropriate for the present purpose. 
A more detailed introduction is provided in  
\cite{H02} and \cite{HHW} (see also \cite[Section 8.5]{SW}, for a description of the underlying ideas). 

Let $K$ be a convex body in $\rd$ and $k\in \{0,\ldots, d-1\}$. Recall that for brevity we write  $\bark=d-1-k$, in the sequel. The set 
$$\nor_k(K)=\{ (x,u,V)\in\partial K\times \Forth(d,\bark):\, (x,u)\in\nor (K)\}$$
is $p$-rectifiable and ${\cal H}^p$-measurable with $p=d-1+k\bark$ (see \cite[Lemma~1]{R02}). 
For $x\in\rd$ and a linear subspace $U\subset\rd$, let $U^\perp$ denote the 
orthogonal complement of $U$, $x|U$ the orthogonal projection of $x$ onto $U$, and $K|U$ 
the orthogonal projection of $K$ onto $U$. Moreover, we write $\partial(K|U)$ for the 
topological boundary of $K|U$ with respect to $U$ as the ambient space. For a given convex body $K$ in $\R^d$ and for a fixed $k\in\{0,\ldots,d-1\}$, 
we consider the projection map 
$$f:(x,u,V)\mapsto (x|V^\perp,V),\qquad (x,u,V)\in  \nor_k(K).$$
The $k$th {\it extended flag measure} $\Gamma_k(K;\cdot)$ of $K$ (for related 
notions, cf.\ \cite{RoZ92,R02,H02}) is a measure on $\rd\times \Forth(d,\bark)$ defined by
\begin{eqnarray*}
\lefteqn{\int g(x,u,V)\, \Ecm_k(K;d(x,u,V))}\\ &=&\tilde{\gamma}(d,k)\int_{G(d,\bark)}\int_{\partial (K|V^\perp)}\left( \sum_{(x,u,V)\in f^{-1}\{ (z,V)\}}g(x,u,V)\right){\cal H}^k(dz)\,\nu_{\bark}^d(dV),
\end{eqnarray*}
where 
$$\tilde{\gamma}(d,k)=\frac 12 \binom{d-1}{ k}
\frac{\Gamma\left(\frac{d-k}2\right)\Gamma\left(\frac{k+1}2\right)} {\Gamma\left(\frac 12\right)\Gamma\left(\frac d2\right)}$$
and $g$ is any bounded measurable function on $\rd\times \Forth(d,\bark)$. 
Note that a result due to Zalgaller \cite{Zal72} implies that $f^{-1}\{ (z,V)\}$ is a singleton, for 
$\nu_{\bark}^d$-almost all $V\in G(d,\bark)$ and ${\cal H}^k$-almost all $z\in\partial (K|V^\perp)$ 
(see \cite[p.\ 89, Corollary 2.3.11]{S} for this and more general results). The normalizing constant 
$\tilde{\gamma}(d,k)$ is chosen such that $\Ecm_k(K;\cdot \times G(d,k^*))=\Xi_k(K;\cdot)$ (see below).

The projection of $\Ecm_k(K;\cdot )$ onto the flag manifold $\Forth(d,\bark)$ will 
be called the $k$th {\it flag measure} $\Omega_k(K;\cdot)$ of $K$; it is given by
\begin{align*}
&\int g(u,V)\, \Eam_k(K;d(u,V))\\
&\qquad\qquad=\tilde\gamma(d,k)\int_{G(d,\bark)}
\int_{\partial (K|V^\perp)}\sum g(u,V)\,{\cal H}^k(dz)\,\nu_{\bark}^d(dV),
\end{align*}
where $g$ is now a bounded measurable function on $F^\perp(d,\bark)$ and 
the summation is extended over all  exterior unit normal vectors $u\in V^\perp\cap S^{d-1}$ 
 of $\partial(K|V^\perp)$ at $z$. If $K|V^\perp$ is $(k+1)$-dimensional, then  $u$ is uniquely determined,  
for ${\cal H}^k$-almost all $z\in  \partial(K|V^\perp)$ (cf.~\cite[p.~73]{S}). If $\dim(\partial(K|V^\perp))=k$, then 
$u$ is unique up to the sign, for ${\cal H}^k$-almost all $z\in  \partial(K|V^\perp)$. 
Finally, if $\dim(\partial(K|V^\perp))<k$, then the inner integral vanishes. 
Thus, by \cite[p.~209 and Theorem 4.2.5]{S}, we obtain
\begin{eqnarray*}
&&\int g(u,V)\, \Eam_k(K;d(u,V))\\
&&\quad =\tilde\gamma(d,k)\int_{G(d,{k+1})}
\int_{S^{d-1}\cap U}g(u,U^\perp)\,S^U_k(K\vert U,du)\,\nu_{{k+1}}^d(dU),
\end{eqnarray*}
where $S_k^U(K\vert U,\cdot)$ 
is the $k$th area measure of the orthogonal projection of $K$ onto $U$, 
with respect to $U$ as the ambient space. Note that this relation holds irrespective 
of the dimension of $K|U$.

Subsequently, we shall use the 
area/coarea formula. 
A suitable version for our purposes can be stated in the following setting. Let $W\subset\R^n$ be $m$-rectifiable, let $Z\subset \R^\nu$ 
be $\mu$-rectifiable, for integers $m\ge \mu\ge 1$, and let $T:W\to Z$ be a Lipschitz map. Then the 
$(\mathcal{H}^m\llcorner W,m)$ approximate $\mu$-dimensional Jacobian of $T$ is denoted by $\text{ap }J_\mu T(w)$ whenever 
$T$ is $(\mathcal{H}^m\llcorner W,m)$ approximately differentiable at $w\in W$. This is the case for $\mathcal{H}^m$-almost all 
$w\in W$. The coarea formula states that for every nonnegative measurable function $g:W\to\R$, we have
$$
\int_W\text{ap }J_\mu T(w)g(w)\,\mathcal{H}^m(dw)=\int_Z\int_{T^{-1}(\{z\})}g(w)\,\mathcal{H}^{m-\mu}(dw)\,\mathcal{H}^{\mu}(dz),
$$
which we shortly summarize as $J_\mu T(w)\,\mathcal{H}^m(dw)=\mathcal{H}^{m-\mu}(dw)\,\mathcal{H}^{\mu}(dz)$. The area formula is  the special case $\mu=m$. 
For more details we refer to \cite[\S 3.2]{Federer69} or \cite[Chapter 3]{Simon}, special versions of the coarea formula 
are described in \cite[Chapter 3]{EG} and \cite[Chapter 5]{KP}. In the following, as in \cite{Simon} we simply write $J_\mu T(w)$ instead of 
the more elaborate notation $\text{ap }J_\mu T(w)$.

We now provide another description of $\Ecm_k(K;\cdot )$. Let $A(d,k)$ denote the affine 
Grassmannian of $k$-dimensional flats (affine subspaces) in $\rd$. Then we define 
$A(K;d,\bark)=f(\nor_k(K))$. Identifying $(z,V)\in A(K;d,\bark)$ 
with $z+V\in A(d,\bark)$, we can interpret $A(K;d,\bark)$ as the set of 
tangent affine $\bark$-flats of $K$.  
Let the projection $P:(z,V)\mapsto V$ be defined on ${A}(K;d,\bark)$. 
By the coarea formula, we thus obtain
(see \cite{Federer69}) 
$$\beta(d,\bark)^{-1}J_{\bark(d-\bark)}P(z,V){\cal H}^p(d(z,V))={\cal H}^k(dz)\nu_{\bark}^d(dV)$$
on $A(K;d,\bark)$. First using this  and  
then the area formula for $f$, we get
\begin{eqnarray*}
&&\int g(x,u,V)\, \Ecm_k(K;d(x,u,V) )\\
&&\qquad =\tilde{\gamma}(d,k)\beta(d,k^*)^{-1}\int_{{A}(K;d,\bark)}\left( \sum_{(x,u,V)\in f^{-1}\{( z,V)\}}g(x,u,V)\right) \\
&&\qquad\qquad \times\, J_{k^*(d-k^*)}P(z,V) \,{\cal H}^p(d(z,V))\\
&&\qquad=\tilde{\gamma}(d,k)\beta(d,\bark)^{-1}\int_{\nor_k(K)}
J_pf(x,u,V)J_{\bark(d-\bark)}P(f(x,u,V))\\
&&\qquad\qquad \times\, g(x,u,V)\, {\cal H}^p(d(x,u,V)),
\end{eqnarray*}
for all bounded measurable functions $g$ on $\rd\times \Forth(d,\bark)$. 
We need a representation of $\Gamma_k(K;\cdot)$ as an integral over the unit normal bundle of $K$. This can be derived from the last expression 
by applying the projection $\Pi: \nor_k(K)\to\nor (K)$, $(x,u,V)\mapsto (x,u)$. The corresponding Jacobians were 
computed in \cite{RoZ92}, and the computation can be summarized by
\begin{eqnarray*}
\lefteqn{J_pf(x,u,V)J_{\bark(d-\bark)}P(f(x,u,V))}\\
&=&J_{d-1}\Pi(x,u,V)\sum_{|I|=\bark}\BK_I(K;x,u)\langle A_I(K;x,u),V\rangle^2.
\end{eqnarray*}
Thus, by another application of the coarea formula, it follows that
\begin{eqnarray} 
&&\int g(x,u,V)\, \Ecm_k(K;d(x,u,V) )\nonumber\\
&&\qquad =\gamma(d,k)\int_{\nor (K)}\sum_{|I|=
\bark}\BK_I(K;x,u)\int_{G^{u^\perp}(d-1,\bark)} g(x,u,V)
\label{ECM}\\
&&\qquad \qquad\times\, \langle V,A_I(K;x,u)\rangle^2\, 
\nu^{d-1}_{\bark}(dV)\,{\cal H}^{d-1}(d(x,u)), \nonumber
\end{eqnarray} 
where 
$$\gamma(d,k)=\frac{\binom{d-1}{k}}{{\cal H}^{\bark}(S^{\bark})}=\tilde{\gamma}(d,k)\frac{\beta(d-1,\bark)}{\beta(d,\bark)}$$
and $G^{u^\perp}(d-1,j)$ is the Grassmannian of $j$-dimensional linear subspaces of $u^\perp$. In the scalar product 
$\langle V,A_I(K;x,u)\rangle^2$, we interpret $V$ and $A_I(K;x,u)$ as one of the two possible associated elements of the 
oriented Grassmannian $G_0^{u^\perp}(d-1,k^*)$. This representation is 
similar to the one for the support measures $\Ecm_k(K;\cdot)$. The crucial difference is that for each $(x,u)$ in the normal bundle of $K$ and for each $I$, the flag measures involve 
an additional averaging  of $g(x,u,V)\langle V,A_I(K;x,u)\rangle^2$ over the linear subspaces $V\in G^{u^\perp}(d-1,\bark)$; these averages are exactly the weights with which the  products $\BK_I(K;x,u)$ of generalized curvatures have to be multiplied.

From this representation it can be seen that the projection $\Pi$ maps $\Ecm_k(K;\cdot)$ to the support measure $\Xi_k(K;\cdot )$. In fact, if $g$ is independent of $V$, then for each $I$, $g(x,u)$ can be removed from the inner integral in \eqref{ECM} and 
the resulting integral then 
is equal to $\binom{d-1}{k}^{-1}$. To verify this, we interpret $V$ and $A_I=A_I(K;x,u)$ as elements of 
$G^{u^\perp}_0(d-1,\bark)$. Then the multivectors $A_I$ with $I\subset\{1,\ldots,d-1\}$ and $|I|=\bark$ form an orthonormal basis of 
$G^{u^\perp}_0(d-1,\bark)$ (cf.\ Section \ref{IoG}) and therefore $\sum_{|I|=\bark}\langle V,A_I\rangle^2=1$. Since 
$$
\int_{G^{u^\perp}(d-1,\bark)}  \langle V,A_I \rangle^2\, 
\nu^{d-1}_{\bark}(dV)
$$
is independent of $I$, the assertion follows.  In particular, we get $\Omega_k(K;\cdot\times G(d,k^*))=\binom{d}{k}(d\kappa_{d-k})^{-1}S_k(K,\cdot)$.

The extended flag measures $\Gamma_k$  also arise naturally, as coefficients  in a Steiner formula 
for affine flats; see, for instance, \cite{H02} and \cite{HHW}.

\section{Integral representation of mixed volumes}\label{results}
Given two convex bodies $K,L$ in $\rd$ and $0\leq k\leq d$, let us denote by
\begin{equation}\label{laterefs}
V_{k,d-k}(K,L)=\binom{d}{ k}V(K[k],-L[d-k])
\end{equation}
a multiple of the mixed volume of $k$ copies of $K$ and $(d-k)$ copies of $-L$. 
These functionals agree with the coefficients in the translative intersection formula for the 
Euler characteristic $V_0$, that is
$$\int_{\rd}V_0(K\cap (L+z))\, {\cal H}^d(dz)=\sum_{k=0}^dV_{k,d-k}(K,L),$$
see \cite{S,SW}. 

Let $k,l\in \{1,\ldots,d-1\}$ be such that $k+l=d$. For the functionals $V_{k,l}$ 
an integral representation has been proved in \cite[Theorem~2]{RZ95} which we shall use subsequently. 
The angle between unit vectors $u,v\in\rd$ is denoted by $\angle(u,v)\in [0,\pi]$. Then we have
\begin{eqnarray}
V_{k,l}(K,L)&=&\int_{\nor (K)\times\nor (L)}F_{k,l}(\angle (u,v))\sum_{|I|=\bark}\sum_{|J|=\barl} \BK_I(K;x,u)\BK_J(L;y,v) \nonumber\\
&&\times \left\|A_I(K;x,u)\wedge u\wedge A_J(L;y,v)\wedge v\right\|^2\, {\cal H}^{2d-2}(d(x,u,y,v)),
\label{IR}
\end{eqnarray}
where  
$$F_{k,l}(\theta )=\frac 1{{\cal H}^{d-1}(S^{d-1})}\frac\theta{\sin\theta}\int_0^1\left(\frac{\sin t\theta}{\sin\theta}\right)^{\bark} \left(\frac{\sin (1-t)\theta}{\sin\theta}\right)^{\barl}\, dt,\qquad \theta\in [0,\pi),$$
and $A_I(K;x,u)$ and $A_J(L;y,v)$ are viewed as multivectors. 
As usual, we put $  0/{\sin 0}=1$. The ratios $\theta/\sin\theta$ and $\sin t\theta/\sin\theta$ remain bounded 
for $\theta\in (0,\pi/2]$, uniformly in $t\in[0,1]$. However, as $\theta$ approaches $\pi$, these expressions  become unbounded. 
So far $F_{k,l}(\pi )$ has not been defined (cf.~also \cite{RZ95}). We can fix $F_{k,l}(\pi )\in[0,\infty)$ arbitrarily, since  $\theta=\pi$ corresponds to $u=-v$, and in this case we have   
$\left\|A_I(K;x,u)\wedge u\wedge A_J(L;y,v)\wedge v\right\|=0$ in \eqref{IR}. 

In addition, we introduce the bounded approximations
$$F_{k,l}^{(\ep)}(\theta)=F_{k,l}(\theta){\bf 1}\{0\leq \theta\leq\pi-\ep\}, \qquad \ep >0,\ \theta\in [0,\pi),$$
and
\begin{eqnarray}
V_{k,l}^{(\ep)}(K,L)&=&\int_{\nor (K)\times\nor (L)}F_{k,l}^{(\ep)}
(\angle (u,v))\sum_{|I|=\bark}\sum_{|J|=\barl} \BK_I(K;x,u)\BK_J(L;y,v)\nonumber \\
&&\times \left\|A_I(K;x,u)\wedge u\wedge A_J(L;y,v)\wedge v\right\|^2\, {\cal H}^{2d-2}(d(x,u,y,v)),\label{7late}
\end{eqnarray}
so that
$$F_{k,l}^{(\ep)}\nearrow F_{k,l}\quad \mbox{ and }\quad V_{k,l}^{(\ep)}(K,L)\nearrow V_{k,l}(K,L),
\quad \text{as }\ep\searrow 0.$$
Here and in the following the symbols $\nearrow$ and $\searrow$ indicate that the limit is approached via an increasing,  respectively a decreasing sequence.

For the bounded approximations of the mixed volumes of two convex bodies, we obtain the following integral representations 
in terms of the  flag measures of the bodies involved.

\begin{Theorem}  \label{main}
Let $k,l\in\{1,\ldots,d-1\}$ and $k+l=d$. Then there exists a continuous function 
$\varphi^{k,l}$ on $F^\perp(d,\bark)\times F^\perp(d,\barl)$ such that
\begin{equation} \label{IR1}
V_{k,l}^{(\ep)}(K,L)=\iint F_{k,l}^{(\ep)}(\angle(u,v))\varphi^{k,l}(u,U,v,V)\, \Eam_k(K;d(u,U))\,\Eam_l(L;d(v,V))
\end{equation}
for arbitrary convex bodies $K,L\subset\R^d$ and $\ep>0$.
\end{Theorem}

Note that Theorem \ref{main} implies
$$
 V_{k,l}(K,L)=\lim_{\ep\searrow 0}\iint F_{k,l}^{(\ep)}(\angle(u,v))\varphi^{k,l}(u,U,v,V)\, \Eam_k(K;d(u,U))\,\Eam_l(L;d(v,V)).
$$

It is natural to ask whether here the limit can be exchanged with the double integral to obtain
\begin{equation} \label{IR2}
V_{k,l}(K,L)=\iint F_{k,l}(\angle(u,v))\varphi^{k,l}(u,U,v,V)\, \Eam_k(K;d(u,U))\,\Eam_l(L;d(v,V)).
\end{equation}
Since $\varphi^{k,l}$ is a signed function and $F_{k,l}$ is unbounded, 
the existence of the integral on the right-hand side is not guaranteed in general. 
 In the final section, we show that if $K=L$ is a $2$-dimensional unit square in $\R^4$, then the integral 
 on the right-hand side of \eqref{IR2} does not exist. However, 
 equation \eqref{IR2} holds under additional assumptions on $K$ and $L$, which (intuitively speaking) 
exclude parallel segments in the boundaries of $K$ and $L$. 

 It seems to be appropriate to include a more detailed comparison of the formulas \eqref{IR} and \eqref{IR2} 
 (provided the latter holds).  
Both formulas relate mixed volumes of two convex bodies $K,L$ to integrals over product spaces. In the case of 
formula \eqref{IR}, the integration extends over the cartesian product of the normal bundles of $K$ and $L$ 
and is carried out with respect to the product of the corresponding Hausdorff measures. In contrast, 
the domain of integration in \eqref{IR2} is independent of $K$ and $L$ and  
a product of flag manifolds. Here the integration is carried out with respect to the product of suitable 
(nonnegative and translation invariant) flag measures of $K$ and $L$.  The integrand on the right-hand side 
of \eqref{IR2} is the product of a  signed function $\varphi^{k,l}(u,U,v,V)$ of the flags $(u,U)$ and $(v,V)$ from the corresponding flag manifolds 
and an unbounded, nonnegative function $F_{k,l}(\angle(u,v))$. In particular, the integrand is independent of $K$ and $L$. 
 On the other hand, the integrand on the right-hand side 
of \eqref{IR} involves  the generalized curvatures of $K$ and $L$. Due to the 
factor $\left\|A_I(K;x,u)\wedge u\wedge A_J(L;y,v)\wedge v\right\|^2$, the double sum under the integral does not factorize, in general. However, since 
the generalized curvature functions are nonnegative, the integral always exists. 

A similar comparison can be given for \eqref{7late} and \eqref{IR1}. In this case, the use of the $\epsilon$-approximation $F_{k,l}^{(\epsilon)}$  ensures that the integral  in \eqref{IR1} exists.

The following theorem states that equation \eqref{IR2} is satisfied, for instance, if at least 
one of the two convex bodies $K$ or $L$ is  ``randomly rotated and/or reflected''. 
Here a random rotation and/or reflection refers to the (unique) invariant probability measure $\nu_d$ on the 
orthogonal group $O(d)$.  

Another condition which ensures that \eqref{IR2} holds is that $K$ and $L$ are 
convex polytopes in general relative position. To define this notion, let $\mathcal{F}_k(K)$ 
denote the set of $k$-dimensional faces of a convex polytope $K$, and let $L(F)$ denote the 
linear subspace parallel to $F\in \mathcal{F}_k(K)$. 
Then we say that convex polytopes $K,L\subset\rd$ 
are in general relative position if $L(F)\cap L(G)=\{o\}$ whenever $F\in \mathcal{F}_k(K)$, $G\in \mathcal{F}_l(L)$  and $k,l\in\{1,\ldots,d-1\}$ with $k+l=d$.

Moreover, we show that \eqref{IR2} holds if the support function of one of the convex bodies $K,L$ 
is of class $C^{1,1}$ (differentiable and the gradient is a 1-Lipschitz map). In this case, the corresponding convex body is strictly convex. 
The following lemma summarizes equivalent conditions for a convex body $K$ to have a support function of class $C^{1,1}$. 
Here we say that $K$ rolls freely (equivalently, slides freely) inside a ball if there is a Euclidean ball 
$B$ such that for each $x\in \partial B$  there is a translation vector $t\in\R^d$ such that $x\in K+t\subset B$. Also, 
$K$ is a summand of a ball  if there is a Euclidean ball $B$ and a convex body $M$ such that $K+M=B$. 

\begin{Lemma}\label{newlem}
Let $K$ be a convex body in $\R^d$. Then the following conditions are equivalent.
\begin{enumerate}
\item[{\rm (a)}] The support function $h_K$ is of class $C^{1,1}$.
\item[{\rm (b)}] The first area measure $S_1(K,\cdot)$ of $K$ is absolutely continuous with bounded density  with respect to spherical Lebesgue measure.
\item[{\rm (c)}] $K$ rolls freely (slides freely) inside a ball.
\item[{\rm (d)}] $K$ is a summand of a ball.
\end{enumerate}
\end{Lemma}

\begin{proof}
The equivalence of (b) and (d) is contained in Theorem 4.7 in \cite{Weil73}. A more general result is 
provided in \cite[Theorem 1]{Weil82}. 

The equivalence of (c) and (d) is well known (see \cite[Theorem 3.2.2]{S}). 

The equivalence of (a) and (b) is, e.g., stated as Proposition 2.3 in \cite{Howard}.
\end{proof}

The following theorem provides sufficient conditions for \eqref{IR2} to hold.

\begin{Theorem}\label{thm3}
Let $K,L\subset\rd$  be arbitrary convex bodies in $\rd$, and let 
$k,l\in\{1,\ldots,d-1\}$ with $k+l=d$. 
Then \eqref{IR2} holds 
\begin{enumerate}
\item[{\rm (a)}] for $K$ and $\rho L$, for $\nu_d$-almost all $\rho\in O(d)$;
\item[{\rm (b)}] if the support function of $K$ or $L$ is of class $C^{1,1}$;
\item[{\rm (c)}] if $K$ and $L$ are polytopes in general relative position. 
\end{enumerate}
\end{Theorem}

The proofs of Theorem \ref{main} and Theorem \ref{thm3} will be given in Sections 6 and 7.

\section{Integrals over Grassmannians}\label{IoG}

In order to obtain the representation \eqref{IR1}, we need to connect equations \eqref{7late} and \eqref{ECM}. This requires a series of preparatory results on integrals over Grassmannians which we provide in this section.

Let an integer $k\in \{0,\ldots, d\}$ and a subspace $A\in G(d,k)$ be fixed. 
By $G_0^A(k,j)$ we denote the set of unit simple $j$-vectors in $A$, where $j\in \{0,\ldots, k\}$.  
For integers $r,s$, we put $r\wedge s:=\min\{r,s\}$ (there is no danger of confusing this notation with the 
exterior product). 
Then, for  $i\in\{0,\ldots ,k\wedge (d-k)\}$, we define the linear subspace 
$$T_iA=\Lin \{\xi\wedge\eta :\, \xi\in G_0^A(k,k-i),\eta\in G_0^{A^\perp}(d-k,i)\}$$
of $\bigwedge_k\rd$. In the following,  
we sometimes also consider $A\in G(d,k)$  as an element of $G_0(d,k)$, which implies that 
one of two possible orientations has to be chosen. In such a case, this choice will not affect the construction. 
If $ a_1,\ldots ,a_d$ is an orthonormal basis of 
$\rd$ such that $A=\Lin\{ a_1,\ldots ,a_k\}$, then 
\begin{equation}\label{ONB}
\left\{\textstyle{\bigwedge}_{j\in I}a_j:\, I\subset\{1,\ldots,d\},|I|=k,|I\cap\{ 1,\ldots ,k\}|=k-i\right\}
\end{equation}
is an orthonormal basis of $T_iA$. In particular, we have 
$T_0A=\Lin\{ a_1\wedge\ldots\wedge a_k\}$ and 
$$\dim (T_iA)=\binom{k}{ i}\binom{d-k}{ i}=:d(i,k).$$
Note  that $T_iA\perp T_jA$ if $i\neq j$ and  
\begin{equation} \label{part}
\textstyle{\bigwedge}_k\rd =\bigoplus\limits_{i=0}^{k\wedge (d-k)}T_iA.
\end{equation}

Given two subspaces $A,B\in G(d,k)$, we define the $i$th product of $A$ and $B$ as
$$\langle A,B\rangle_i=\left\| p_{T_iA}B\right\|,$$
where $p_{T_iA}B$ denotes the orthogonal projection of $B$ (that is, of a  simple unit 
$k$-vector $B_0$ corresponding to $B$) onto $T_iA$. 
More explicitly, if $B_0\in G_0(d,k)$ corresponds to $B$ and $\eta_1,\ldots,\eta_{d(i,k)}$ is 
an orthonormal basis of $T_iA$, then 
$$
\|p_{T_iA}B\|^2=\sum_{r=1}^{d
(i,k)}\langle B_0,\eta_r\rangle^2.
$$
In particular, we obtain 
$$\langle A,B\rangle_0=\left\| p_{T_0A}B\right\| =|\langle A,B\rangle|.$$
The expression $|\langle A,B\rangle|$ can be taken as the absolute value of the scalar product 
of  simple unit $k$-vectors corresponding to $A$ and $B$ or as the absolute value of the determinant 
of the orthogonal projection of $A$ onto $B$, which yields the same numerical value.

Furthermore, the $i$th product is symmetric, i.e., we have $\langle A,B\rangle_i=\langle B,A\rangle_i$, as follows from the subsequent lemma. We include a proof, 
since we could not find an explicit reference.  

\begin{Lemma}
Let $k\in \{0,\ldots, d\}$ and $A,B\in G(d,k)$. Then there is an orthogonal map  $\varrho\in O(d)$ such 
that $\varrho A=B$ and $\varrho B=A$.
\end{Lemma}

\begin{proof}
For the proof, we can assume that $A\cap B=\{0\}$. In fact, otherwise let $L_0=A\cap B$. Then 
we define $\varrho$ as the identity on $L_0$. It then remains to consider  $A\cap L_0^\perp$ 
and $B\cap L_0^\perp$ in $L_0^\perp$, for which we have $(A\cap L_0^\perp)\cap (B\cap L_0^\perp)=\{0\}$. 

The assertion of the lemma with $A\cap B=\{0\}$ is proved by induction with respect to $k\ge 0$. 
For $k=0$ there is nothing to show. If $k=1$, let $A=\text{Lin}\{a\}$ and $B=\text{Lin}\{b\}$ 
with $a,b\in S^{d-1}$. We define $\varrho$ on $L=\text{Lin}\{a,b\}$ as the orthogonal reflection 
which interchanges $a$ and $b$, and on $L^\perp$ as the identity map, which yields the required isometry.

Now assume that $k\ge 2$ and that the assertion is true for all integers smaller than $k$. Clearly, 
there exist $a_1\in A\cap S^{d-1}$ and $b_1\in B\cap S^{d-1}$ such that
$$
\|a_1-b_1\|=\min\{\|a-b\|:a\in A\cap S^{d-1}, b\in B\cap S^{d-1}\}>0.
$$
We put $L=\text{Lin}\{a_1,b_1\}$  and have $\text{dim}(L)=2$. Then, for $a\in A\cap a_1^\perp$ 
and $b\in B\cap b_1^\perp$, it follows that $ \langle a,  b_1\rangle =0$ and $\langle b, a_1\rangle =0$. 

In fact, let $b\in B\cap b_1^\perp\cap S^{d-1}$ be arbitrarily chosen. Then, for $\theta\in (-\pi,\pi)$, 
we have $\cos(\theta) b_1+\sin(\theta) b\in B\cap S^{d-1}$, and therefore
$$
f(\theta)=\|a_1-(\cos(\theta) b_1+\sin(\theta) b)\|^2
$$
attains its minimum for $\theta=0$. Thus $f'(0)=0$, which implies that
$$
0=2\langle a_1-b_1,-b\rangle =-2 \langle a_1, b\rangle,
$$
and this yields the second assertion. The first assertion follows by symmetry. 

Hence,   we have 
$A\cap a_1^\perp,B\cap b_1^\perp\subset L^\perp$ and 
$\text{dim}(A\cap a_1^\perp)=\text{dim}(B\cap b_1^\perp)=k-1$. By induction, 
there is an isometry $\varrho_1$ of $L^\perp$ which interchanges $A\cap a_1^\perp$ and $B\cap b_1^\perp$. 
The induction is completed by defining $\varrho$ on $L^\perp$ as $\varrho_1$ and 
on $L$ as the orthogonal reflection which interchanges $a_1$ and $b_1$.
\end{proof}

In the following two lemmas we evaluate certain integrals over Grassmannians. These lemmas are 
needed in Section \ref{sec:6} to construct the solution of an integral equation.

\begin{Lemma} \label{L1}
There exist positive constants $c_{k,0}^d,\ldots ,c_{k,k\wedge (d-k)}^d$ such that 
for any two subspaces $A,B\in G(d,k)$,
$$\int_{G(d,k)}\langle A,V\rangle^2\langle V,B\rangle^2\, \nu^d_k(dV)=
\sum_{i=0}^{k\wedge (d-k)}c_{k,i}^d\langle A,B\rangle_i^2.$$
\end{Lemma} 

\begin{proof}
Since the case $k\in\{0,d\}$ is trivial, we assume that $k\in\{1,\ldots,d-1\}$ subsequently. 
For the subspace $A\in G(d,k)$, we choose an orthonormal basis 
 $a_1,\ldots,a_d$ of $\rd$ such that $A=\Lin\{a_1,\ldots,a_k\}$ (as at the beginning of this section).  
In the above formula, we can equivalently represent the subspaces $A,B$ by 
elements of $G_0(d,k)$ and integrate with respect to the $O(d)$ invariant probability measure $\bar{\nu}_k^d$ on $G_0(d,k)$.
Further, due to \eqref{part}, we can write $B=\sum_ip_{T_iA}B$, where here and in the following 
all sums over $i$ will run from $0$ to $k\wedge (d-k)$. Let $\bar\xi_i\in \bigwedge_k\R^d$ be a 
unit $k$-vector such that $p_{T_iA}B=\|p_{T_iA}B\|\cdot \bar\xi_i$. 
Then we obtain 
\begin{equation}\label{expanddistr}
\langle V,B\rangle^2=\sum_i\left\| p_{T_iA}B\right\|^2\langle V,\bar\xi_i\rangle^2
+\sum_{i\neq j}\left\| p_{T_iA}B\right\|\, \left\| p_{T_jA}B\right\|\, \langle V,
\bar\xi_i\rangle \, \langle V,\bar\xi_j\rangle.
\end{equation}

{\bf Claim 1:} Let $\xi_i\in T_iA$ and $\xi_j\in T_jA$. If $i\neq j$ or if $i=j$ and $\xi_i,\xi_j$  are different 
elements of the orthonormal basis \eqref{ONB}, then
$$
\int_{G_0(d,k)}\langle A,V\rangle^2\langle V,\xi_i\rangle\langle V,\xi_j\rangle \, \bar\nu^d_k(dV)=0.
$$
Since the above integral is linear in $\xi_i,\xi_j$, it is sufficient to consider the case where the multivectors 
$\xi_i,\xi_j$ are different simple $k$-vectors from the bases described in \eqref{ONB}. 
In this case, there are sets $I,J\subset\{1,\ldots,d\}$ 
with $|I|=|J|=k$, $|I\cap\{1,\ldots,k\}|=k-i$, $|J\cap\{1,\ldots,k\}|=k-j$ such that 
$$
\xi_i=\textstyle{\bigwedge}_{l\in I}a_l,\qquad \xi_j=\textstyle{\bigwedge}_{l\in J}a_l.
$$
Since $I\neq J$, we can fix an index  $\iota\in (I\setminus J)\cup(J\setminus I)$. Let 
$\varrho\in O(d)$ be defined by $\varrho(a_\iota)=-a_\iota$ and $\varrho(a_l)=a_l$ for $l\neq \iota$. The $O(d)$ invariance of 
$\bar\nu^d_k$ then implies that
\begin{align*}
&\int_{G_0(d,k)}\langle A,V\rangle^2\langle V,\xi_i\rangle\langle V,\xi_j\rangle \, \bar\nu^d_k(dV)\\
&\qquad =\int_{G_0(d,k)}\langle A,\varrho^{-1}V\rangle^2\langle \varrho^{-1}V,\xi_i\rangle\langle \varrho^{-1}V,\xi_j\rangle \, \bar\nu^d_k(dV)\\
&\qquad =\int_{G_0(d,k)}\langle \varrho A, V\rangle^2\langle  V,\varrho\xi_i\rangle\langle V,\varrho\xi_j\rangle \, \bar\nu^d_k(dV)\\
&\qquad =-\int_{G_0(d,k)}\langle A,V\rangle^2\langle V,\xi_i\rangle\langle V,\xi_j\rangle \, \bar\nu^d_k(dV),
\end{align*}
since $\langle \varrho A, V\rangle^2=\langle A,V\rangle^2$ and 
$\langle  V,\varrho\xi_i\rangle\langle V,\varrho\xi_j\rangle=
-\langle V,\xi_i\rangle\langle V,\xi_j\rangle$. 
This establishes Claim 1.

From Claim 1 and \eqref{expanddistr} we now conclude that 
$$
\int_{G_0(d,k)}\langle A,V\rangle^2\langle V,B\rangle^2\, \bar\nu^d_k(dV)
 =\sum_i \langle A,B\rangle_i^2\int_{G_0(d,k)}\langle A,V\rangle^2\langle V,\bar\xi_i\rangle^2\, \bar\nu^d_k(dV).
$$
The assertion of the lemma follows immediately from this and the subsequent claim.

{\bf Claim 2:} The integral 
$$
\int_{G_0(d,k)}\langle A,V\rangle^2\langle V,\xi_i\rangle^2\, \bar{\nu}_k^d(dV)
$$
is independent of $A\in G_0(d,k)$ and  $\xi_i\in T_iA$ with $\|\xi_i\|=1$, 
and is therefore a constant $c^d_{k,i}$. To prove this, it is sufficient to show that the integral is independent of $\xi_i\in T_iA$. The independence 
of $A$ then follows from the $O(d)$ invariance of $\bar\nu^d_k$. Since 
$\xi_i\in T_iA$ and $\|\xi_i\|=1$, there are $\alpha_1,\ldots,\alpha_{d(i,k)}\in \R$ and $k$-vectors $\eta_1,\ldots,\eta_{d(i,k)}\in T_iA$ of the orthonormal basis \eqref{ONB} such that 
$$
\xi_i=\alpha_1\eta_1+\cdots +\alpha_{d(i,k)}\eta_{d(i,k)}
\qquad\text{and}\qquad\alpha_1^2+\ldots+\alpha_{d(i,k)}^2=1.
$$ Since
$$
\langle V,\xi_i\rangle^2=\sum_{l=1}^{d(i,k)}\alpha_l^2\langle V,\eta_l\rangle^2+\sum_{r\neq s}\alpha_r\alpha_s\langle V,\eta_r\rangle 
\langle V,\eta_s\rangle,
$$
it follows from Claim 1 that 
$$\int_{G_0(d,k)} \langle A,V\rangle^2\langle V,\xi_i\rangle^2\,  \bar{\nu}_k^d(dV)=\sum_{l=1}^{d(i,k)}\alpha_l^2\int_{G_0(d,k)}\langle A,V\rangle^2\langle V,\eta_l\rangle^2\, \bar{\nu}_k^d(dV).
$$
Finally, we observe that 
$$
\int_{G_0(d,k)}\langle A,V\rangle^2\langle V,\eta_r\rangle^2\, \bar{\nu}_k^d(dV)
$$
is independent of $r\in\{1,\ldots,d(i,k)\}$. To see this, let $r,s\in\{1,\ldots,d(i,k)\}$. Since $\eta_r,\eta_s\in T_iA$, there is 
some $\varrho\in O(d)$ with $\varrho A=A$ and $\varrho\eta_r=\eta_s$. Then the assertion follows again from the $O(d)$ invariance 
of $\bar\nu^d_k$, since 
\begin{align*}
\int_{G_0(d,k)}\langle A,V\rangle^2\langle V,\eta_r\rangle^2\, \bar{\nu}_k^d(dV)
& =\int_{G_0(d,k)}\langle A,\varrho^{-1}V\rangle^2\langle \varrho^{-1}V,\eta_r\rangle^2\, \bar{\nu}_k^d(dV)\\
& =\int_{G_0(d,k)}\langle \varrho A,V\rangle^2\langle V,\varrho\eta_r\rangle^2\, \bar{\nu}_k^d(dV)\\
& =\int_{G_0(d,k)}\langle  A,V\rangle^2\langle V,\eta_s\rangle^2\, \bar{\nu}_k^d(dV).
\end{align*} 
This completes the proof of Claim 2.
\end{proof}

\begin{Lemma} \label{L2}
There exist positive constants $d_{i,j}^{d,k}$, $i,j=0,\ldots ,k\wedge (d-k)$, such that for any two subspaces $A,B\in G(d,k)$,
$$\int_{G(d,k)}\langle A,V\rangle_i^2\langle V,B\rangle^2\, \nu^d_k(dV)=\sum_{j=0}^{k\wedge (d-k)}d_{i,j}^{d,k}\langle A,B\rangle_j^2.$$
\end{Lemma} 

\begin{proof} Since the case $k\in\{0,d\}$ is trivial, we assume that $k\in\{1,\ldots,d-1\}$ in the following. 
Let $A,B\in G(d,k)$ be two linear subspaces. 
We fix an orthonormal basis $ a_1,\ldots ,a_d$  of $\rd$ such that $A=\Lin\{ a_1,\ldots ,a_k\}$, and put $I_0=\{ 1,\ldots ,k\}$. Then
$$\langle A,V\rangle_i^2=\sum_{\substack{|I|=k\\|I\cap I_0|=k-i}}\langle A_I,V\rangle^2,$$
where $A_I=\Lin\{ a_i:\, i\in I\}$ and $V\in G(d,k)$. Thus, Lemma \ref{L1} implies that
\begin{align*}
&\int_{G(d,k)}\langle A,V\rangle_i^2\langle V,B\rangle^2\,{\nu}_k^d(dV)\\
&\qquad=
\sum_{\substack{|I|=k\\|I\cap I_0|=k-i}} \int_{G(d,k)}\langle A_I,V\rangle^2\langle V,B\rangle^2\, {\nu}_k^d(dV)\\
&\qquad=\sum_{\substack{|I|=k\\|I\cap I_0|=k-i}} \sum_{m=0}^{k\wedge (d-k)}c^d_{k,m}\langle A_I,B\rangle_m^2\\
&\qquad=\sum_{\substack{|I|=k\\|I\cap I_0|=k-i}} \sum_{m=0}^{k\wedge (d-k)}c^d_{k,m} \sum_{\substack{|J|=k\\|J\cap I|=k-m}}\langle A_J,B\rangle^2\\
&\qquad=\sum_{\substack{|I|=k\\|I\cap I_0|=k-i}} \sum_{m=0}^{k\wedge (d-k)}\sum_{j=0}^{k\wedge (d-k)}
\sum_{\substack{|J|=k\\|J\cap I|=k-m\\|J\cap I_0|=k-j}}
c^d_{k,m}  \langle A_J,B\rangle^2\\
&\qquad= \sum_{j=0}^{k\wedge (d-k)}\sum_{\substack{|J|=k\\|J\cap I_0|=k-j}}\langle A_J,B\rangle^2 \sum_{m=0}^{k\wedge (d-k)}c^d_{k,m} 
\sum_{\substack{|I|=k\\|I\cap I_0|=k-i\\|I\cap J|=k-m}}1.
\end{align*}
Since the cardinality of the set of all $I\subset\{1,\ldots,d\}$ with $|I|=k$, $|I\cap I_0|=k-i$ and $|I\cap J|=k-m$ does 
not depend on the particular choice of the index set $J$ satisfying $|J\cap I_0|=k-j$, we define 
\begin{equation}\label{combrel}
d_{i,j}^{d,k}=\sum_{m=0}^{k\wedge (d-k)}c^d_{k,m}\cdot\card\{ I:\, |I|=k,|I\cap I_0|=k-i,|I\cap J|=k-m\}.
\end{equation}
Then we obtain 
\begin{align*}
\int_{G(d,k)}\langle A,V\rangle_i^2\langle V,B\rangle^2\,{\nu}_k^d(dV)
&=\sum_{j=0}^{k\wedge (d-k)} d_{i,j}^{d,k}\sum_{\substack{|J|=k\\|J\cap I_0|=k-j}}\langle A_J,B\rangle^2\\
&=\sum_{j=0}^{k\wedge (d-k)}d_{i,j}^{d,k}\langle A,B\rangle_j^2,
\end{align*}
which completes the proof.
\end{proof}

\bigskip

\noindent
{{\bf Remark.}} The cardinality on the right-hand side of equation \eqref{combrel}  
can be expressed explicitly in terms of binomial coefficients. Introducing the variable $l=|I\cap I_0\cap J|$, we get
$$
d_{i,j}^{d,k}=\sum_{m=0}^{k\wedge (d-k)}c^d_{k,m}\cdot\sum_{l=0}^{k}
\binom{k-j}{l}\binom{j}{k-i-l}\binom{j}{k-m-l}\binom{d-k-j}{m+l+i-k},
$$
where $\binom{a}{b}=0$ for integers $a,b$ with $a\ge 0$, if $b<0$ or $b>a$.

We now show that the matrix of coefficients from Lemma~\ref{L2}, that is 
$$D(d,k)=\left(d^{d,k}_{i,j}\right)_{i,j=0}^{k\wedge (d-k)},$$
is regular.


Let an orthonormal basis $ e_1,\ldots ,e_d$ of $\rd$ be fixed. For the given basis, we define  $E_I=\Lin\{ e_i:\, i\in I\}$, where $I\subset\{ 1,\ldots ,d\}$.
Let $\cal L$ be the $\binom{d}{ k}$-dimensional linear space of continuous functions on $G(d,k)$ spanned by the functions $\langle E_I,\cdot\rangle^2$ with $|I|=k$. We equip 
$\cal L$  with the scalar product
$$(f,g)=\int_{G(d,k)}f(U)g(U)\,\nu_k^d(dU)$$
and consider the linear operator $T:\mathcal{L}\to\mathcal{L}$  given by 
$$T: g\mapsto \int_{G(d,k)}g(U)\langle U,\cdot\rangle^2\, \nu_k^d(dU).$$
Lemma \ref{L1} implies that $T$ is well defined. 
By Fubini's theorem,  $T$ is self-adjoint. 
Hence there exist pairwise different real eigenvalues $\alpha_1,\ldots,\alpha_m$ of $T$ with  corresponding orthogonal eigenspaces $H_1,\ldots,H_m$ such that
${\cal L}=H_1\oplus\cdots\oplus H_m$.

\begin{Lemma}
The operator $T$ is surjective on ${\cal L}$.
\end{Lemma}

\begin{proof}  
We fix an index set $I$ with $|I|=k$ and show that 
$\langle E_I,\cdot\rangle^2\in T{\cal L}$. Due to the above observations, we can write
\begin{equation} \label{sum}
\langle E_I,\cdot\rangle^2=h_1+\cdots +h_m
\end{equation}
with $h_i\in H_i$, $1\leq i\leq m$. Then we get
\begin{eqnarray*}
\alpha_ih_i(E_I)&=&\int_{G(d,k)} h_i(U) \langle E_I,U\rangle^2\, \nu_k^d(dU)\\
&=&\int_{G(d,k)} \sum_{j=1}^m h_i(U)h_j(U)\, \nu_k^d(dU)\\
&=&\sum_{j=1}^m \int_{G(d,k)} h_i(U)h_j(U)\, \nu_k^d(dU)\\
&=&\sum_{j=1}^m (h_i,h_j)=(h_i,h_i).
\end{eqnarray*}
If $\alpha_i=0$, for some $i\in\{1,\ldots,m\}$, then $h_i=0$ in \eqref{sum}. 
Therefore we can assume that $\alpha_i\neq 0$ for all functions $h_i$ in \eqref{sum}. Choosing  $h=\sum_i\alpha_i^{-1}h_i$, we obtain that
$$Th=\sum_i\alpha_i^{-1}Th_i=\sum_i h_i=\langle E_I,\cdot\rangle^2,$$
and the proof is complete.
\end{proof}

\begin{Corollary}\label{Cor1}
The operator $T:{\cal L}\to{\cal L}$ is bijective and all its eigenvalues are nonzero.
\end{Corollary}

Fix $I\subset\{1,\ldots,d\}$ with $|I|=k$  and consider the linear space
$${\cal L}_I=\Lin\{\langle E_I,\cdot\rangle^2_i : i=0,\ldots ,k\wedge(d-k)\}.$$
By Lemma~\ref{L2}, $T$ maps ${\cal L}_I$ into ${\cal L}_I$. 
By Corollary \ref{Cor1}, $T$ is injective, hence $T|{\cal L}_I$ is  injective and therefore a bijection. Since $D(d,k)$ is the matrix of the restriction $T|{\cal L}_I$, 
we obtain the regularity of $D(d,k)$.

\begin{Proposition} \label{regul}
The matrix $D(d,k)$ is  regular.
\end{Proposition} 

\section{Proof of Theorem \ref{main}}\label{sec:6}

In this section, we always assume that $k,l\in\{1,\ldots,d-1\}$ and $k+l=d$. 

Comparing the representation \eqref{7late} with \eqref{IR1} and \eqref{ECM}, we see that Theorem~\ref{main}
is a consequence of the following result.

\begin{Proposition}  \label{P-dint}
There exists a continuous function $\varphi^{k,l}$ on $F^\perp(d,\bark)\times 
F^\perp(d,\barl)$ such that 
\begin{eqnarray*}
&&\int_{G^{u^\perp}(d-1,\bark)}\int_{G^{v^\perp}(d-1,\barl)}\langle A,U\rangle^2\varphi^{k,l}(u,U,v,V)\langle V,B\rangle^2\, \nu^{d-1}_{\barl}(dV)\nu^{d-1}_{\bark}(dU)\\
&&\qquad=\frac 1{\gamma(d,k)\gamma(d,l)}\,\|A\wedge u\wedge B\wedge v\|^2 \nonumber
\end{eqnarray*}
for any $(u,A)\in F^\perp(d,\bark)$ and $(v,B)\in F^\perp(d,\barl)$.
\end{Proposition}

\begin{proof}
Given $0\leq p\leq k\wedge\bark$ and $0\leq q\leq l\wedge\barl$, let us introduce the function
$$\varphi^{k,l}_{p,q}(u,U,v,V)=\sum_{\substack{|I|=\bark\\ |I\cap I_0|=\bark-p}}\;\; 
\sum_{\substack{|J|=
\barl\\|J\cap J_0|=\barl-q}}\left\|\textstyle{\bigwedge}_{i\in I}u_i\wedge u\wedge\textstyle{\bigwedge}_{j\in J}v_j\wedge v\right\|^2$$
on $F^\perp(d,k^*)\times F^\perp(d,l^*)$, 
where $ u_1,\ldots ,u_d$ and $ v_1,\ldots ,v_d$ are orthonormal bases of $\rd$ such that $u_d=u$, $v_d=v$, 
$U=\Lin\{ u_i:\, i\in I_0\}$ and $V=\Lin\{ v_j:\, j\in J_0\}$. Here  $I_0,J_0$ are fixed subsets 
 of $\{1,\ldots,d-1\}$ with $|I_0|=\bark$, $|J_0|=\barl$, and the summation is over subsets 
 $I,J\subset\{1,\ldots,d-1\}$.  First, 
we  show that $\varphi^{k,l}_{p,q}(u,U,v,V)$ does not depend on the choice of the two bases. We start with a simple remark. 
Let $a_1\ldots,a_d$ be an orthonormal basis of $\R^d$ and let $b_1,\ldots,b_s$ denote an orthonormal system in $\R^d$, 
where $r,s\ge 0$ are integers such that $r+s=d$. Expressing $b_1,\ldots,b_s$ in the basis $a_1,\ldots,a_d$, and by 
basic properties of alternating products, we obtain
\begin{align*}
&\|a_1\wedge\ldots\wedge a_r\wedge b_1\wedge\ldots\wedge b_s\|^2\\
&\qquad=\left\|a_1\wedge\ldots\wedge a_r\wedge \sum_{i=r+1}^d\langle b_1,a_i\rangle a_i\wedge \ldots\wedge \sum_{i=r+1}^d\langle b_s,a_i\rangle a_i\right\|^2\\
&\qquad=\left(\det\left(\left(\langle a_i,b_j\rangle\right)_{i=r+1,j=1}^{d,s}\right)\right)^2\\
&\qquad=
\langle a_{r+1}\wedge\ldots\wedge a_d,
b_1\wedge\ldots\wedge b_s\rangle^2.
\end{align*}
Assume that 
$\textstyle{\bigwedge}_{i\in I}u_i\wedge u\wedge v\neq 0$. We denote by $(\textstyle{\bigwedge}_{i
\in I}u_i\wedge u\wedge v)^\perp$ a unit simple $(k-1)$-vector whose associated subspace is orthogonal to the subspace 
associated with $\textstyle{\bigwedge}_{i\in I}u_i\wedge u\wedge v$. Observe that $|I|+1+|J|+1=d-k+d-l=d$. Then we get from the preceding remark that
$$
\left\|\textstyle{\bigwedge}_{i\in I}u_i\wedge u\wedge\textstyle{\bigwedge}_{j\in J}v_j\wedge v\right\|^2
=\left\|\textstyle{\bigwedge}_{i\in I}u_i\wedge 
u\wedge v\right\|^2 \left\langle \left(\textstyle{\bigwedge}_{i\in I}u_i\wedge u\wedge v\right)^\perp,\textstyle{
\bigwedge}_{j\in J}v_j\right\rangle^2.
$$
Since 
$$
\left\{\textstyle{\bigwedge}_{j\in J}v_j:|J|=l^*,|J\cap J_0|=l^*-q\right\}
$$
is an orthonormal basis of $T_qV$,
we obtain
\begin{align*}
&\varphi^{k,l}_{p,q}(u,U,v,V)\\
&\quad=\sum_{\substack{|I|=\bark\\ |I\cap I_0|=\bark-p}}
\left\|\textstyle{\bigwedge}_{i\in I}u_i\wedge 
u\wedge v\right\|^2\sum_{\substack{|J|=\barl\\ |J\cap J_0|=\barl-q}}
\left\langle \left(\textstyle{\bigwedge}_{i\in I}u_i\wedge u\wedge v\right)^\perp,\textstyle{
\bigwedge}_{j\in J}v_j\right\rangle^2
\\
&\quad=\sum_{\substack{|I|=\bark\\ |I\cap I_0|=\bark-p}}
\left\|\textstyle{\bigwedge}_{i\in I}u_i\wedge 
u\wedge v\right\|^2\left\|p_{T_qV}\left(\textstyle{\bigwedge}_{i\in I}u_i\wedge u\wedge v\right)^\perp\right\|^2\\
&\quad=
\sum_{\substack{|I|=\bark\\|I\cap I_0|=\bark-p}}\left\|\textstyle{\bigwedge}_{i\in I}u_i\wedge 
u\wedge v\right\|^2 \left\langle \left(\textstyle{\bigwedge}_{i\in I}u_i\wedge u\wedge v\right)^\perp,V\right\rangle_q^2,
\end{align*}
where the symmetry of the $q$th product was used. Note that the preceding argument can be considered in $\R^d$ or with 
respect to $v^\perp$ as the ambient space. This is also true for the $q$-product which appears in the last equation. 
Moreover,  if 
$\textstyle{\bigwedge}_{i\in I}u_i\wedge 
u\wedge v=0$, then we can choose any $l^*$-dimensional subspace for $\left(\textstyle{\bigwedge}_{i\in I}u_i\wedge u\wedge v\right)^\perp$, since then the right-hand side is zero. 
 This  shows the independence of the choice of the sequence $(v_j)$. By a similar argument, the independence of the choice of the sequence $(u_i)$ is shown.

From Lemma~\ref{L2} we obtain
\begin{eqnarray*}
\lefteqn{\int_{G^{v^\perp}(d-1,\barl)}\varphi^{k,l}_{p,q}(u,U,v,V)
\langle V,B\rangle^2\, \nu^{d-1}_{\barl}(dV)}\\
&=&\sum_{\substack{|I|=\bark\\|I\cap I_0|=\bark-p}}\left\|
\textstyle{\bigwedge}_{i\in I}u_i\wedge u\wedge v\right\|^2\\
&&\qquad \times \int_{G^{v^\perp}(d-1,\barl)}
\left\langle \left(\textstyle{\bigwedge}_{i\in I}u_i\wedge u\wedge v\right)^\perp,V\right\rangle_q^2 \langle V,B\rangle^2\, \nu^{d-1}_{\barl}(dV)\\
&=&\sum_{\substack{|I|=\bark\\|I\cap I_0|=\bark-p}}\left\|\textstyle{\bigwedge}_{i\in I}u_i\wedge u\wedge v\right\|^2 
\sum_{j=0}^{l\wedge\barl}d^{d-1,\barl}_{q,j} \left\langle \left(\textstyle{\bigwedge}_{i\in I}u_i\wedge u\wedge v\right)^\perp,B\right\rangle^2_j\\
&=&\sum_{j=0}^{l\wedge\barl}d^{d-1,\barl}_{q,j} \varphi^{k,l}_{p,j}(u,U,v,B).
\end{eqnarray*}
Using Lemma~\ref{L2} again, we finally get
\begin{eqnarray*} \label{double-d}
&&\int_{G^{u^\perp}(d-1,\bark)}\int_{G^{v^\perp}(d-1,\barl)}\langle A,U\rangle^2\varphi^{k,l}_{p,q}(u,U,v,V)\langle V,B\rangle^2\, \nu^{d-1}_{\barl}(dV)\nu^{d-1}_{\bark}(dU)\, \nonumber\\
&&=\sum_{i=0}^{k\wedge\bark}\sum_{j=0}^{l\wedge\barl}d^{d-1,\bark}_{p,i}d^{d-1,\barl}_{q,j}\varphi^{k,l}_{i,j}(u,A,v,B).
\end{eqnarray*}
Now we prove the existence of coefficients $\alpha_{p,q}\in\R$ such that the system of linear equations
\begin{equation}\label{system}
\sum_{p=0}^{k\wedge\bark}\sum_{q=0}^{l\wedge\barl}\alpha_{p,q} d^{d-1,\bark}_{p,i}d^{d-1,\barl}_{q,j}=\begin{cases}
(\gamma(d,k)\gamma(d,l))^{-1},&i=j=0,\\[1ex]
0,& \mbox{otherwise},
\end{cases}
\end{equation}
is satisfied. If this has been shown, then it follows that the function $\varphi^{k,l}$ given by  
\begin{equation} \label{phi}
\varphi^{k,l} =\sum_{p=0}^{k\wedge\bark}\sum_{q=0}^{l\wedge\barl}\alpha_{p,q}\varphi^{k,l}_{p,q}
\end{equation}
is a solution of our problem.  

The matrix of the linear system \eqref{system} can be written as the Kronecker product of two matrices, $D(d-1,\bark)\otimes D(d-1,\barl)$, where 
the Kronecker product of an $r\times r$ matrix $M=(m_{i,j})_{i,j=1}^r$ and an $s\times s$ matrix $N$ is the $(rs)\times (rs)$ 
matrix defined by 
$$M\otimes N=\left(\begin{array}{ccc}
            m_{11}N&\ldots &m_{1k}N\\
	    \vdots&&\vdots\\
	    m_{k1}N&\ldots &m_{kk}N
	    \end{array}\right) .$$
It is well known that if $M,N$ are two regular square matrices, then $M\otimes N$ is regular as well. 
The system \eqref{system} of 
linear equations is then equivalent to
\begin{align*}
&(\alpha_{0,0},\ldots,\alpha_{0,l\wedge l^*},\alpha_{1,0},\ldots,\alpha_{1,l\wedge  l^*},\ldots,\alpha_{k\wedge\bark,l\wedge\barl})D(d-1,\bark)\otimes D(d-1,\barl)\\
&\qquad =
((\gamma(d,k)\gamma(d,l))^{-1},0,\ldots,0).
\end{align*}
Hence, by Proposition \ref{regul}, the $(k\wedge k^* +1)(l\wedge l^*+1)$ matrix $D(d-1,\bark)\otimes D(d-1,\barl)$ 
is  regular and therefore the above linear system has a unique solution and the proof is complete.
\end{proof}

\medskip

\noindent
{\bf Remark.} The function $\varphi^{k,l}$ constructed in the proof has the symmetry 
property $\varphi^{k,l}(u,U,v,V)=\varphi^{l,k}(v,V,u,U)$ for all $(u,U)\in F^\perp(d,k^*)$ and $(v,V)\in F^\perp(d,l^*)$.

\medskip

\section{Proof of Theorem \ref{thm3}}

(a) Let $\rho\in O(d)$. Relation \eqref{IR1} holds for all $\ep>0$ and arbitrary 
convex bodies $K,L\subset\rd$. Since $ V_{k,l}^{(\ep)}(K,\rho L)\nearrow V_{k,l}(K,\rho L) $ 
as $\ep\searrow 0$, we have to show that 
\begin{align*}
&\iint F_{k,l}^{(\ep)}(\angle (u,v))\varphi^{(k,l)}(u,U,v,V) \, 
\Omega_k(K;d(u,U))\, \Omega_l(\rho L;d(v,V))\\
&\qquad \to \iint F_{k,l}(\angle (u,v))\varphi^{(k,l)}(u,U,v,V) \, 
\Omega_k(K;d(u,U))\, \Omega_l(\rho L;d(v,V))
\end{align*}
as $\ep\searrow 0$, for $\nu_d$-almost all $\rho\in O(d)$. 

Since $F_{k,l}^{(\ep)}\nearrow F_{k,l}$ as $\ep\searrow 0$, the result follows if the dominated 
convergence theorem can be applied. To justify this, observe that, for $\theta\in[0,\pi)$ and 
$(u,U,v,V)\in F^\perp(d,k^*)\times F^\perp(d,l^*)$,
\begin{align*}
0\le F_{k,l}^{(\ep)}(\theta)&\le F_{k,l}(\theta)\leq\mbox{const}\cdot {\sin^{1-d}\theta},\\
|\varphi^{k,l}(u,U,v,V)|&\leq\mbox{const}\cdot |\sin\angle(u,v)|^2.
\end{align*}
For the latter inequality, we use \eqref{phi} and the general estimate $\|\xi\wedge\eta\|\le \|\xi\|\cdot \|\eta\|$ which 
holds for arbitrary simple multivectors (cf.\ \cite[p.\ 32]{Federer69} where, however, the norm is denoted by $|\cdot|$), 
and implies that
\begin{align*}
\left\|\textstyle{\bigwedge}_{i\in I}u_i\wedge u\wedge\textstyle{\bigwedge}_{j\in J}v_j\wedge v\right\|^2
&=\left\|\textstyle{\bigwedge}_{i\in I}u_i\wedge  \textstyle{\bigwedge}_{j\in J}v_j\wedge u\wedge v\right\|^2\\
&\le\|u\wedge v\|=|\sin\angle (u,v)|^2.
\end{align*}

Thus we have
$$
F_{k,l}^{(\ep)}(\angle(u,v))\,|\varphi^{k,l}(u,U,v,V)|\le \mbox{const}\cdot \sin^{3-d}\angle(u,v).
$$
It remains to show that 
$$
\iint \sin^{3-d}\angle(u,v)\,\Eam_k(K;d(u,U))\,\Eam_l(\rho L;d(v,V))<\infty,
$$
for $\nu_d$-almost all $\rho\in O(d)$. To see this, we use Fubini's theorem, the fact that 
(up to constants) the area measures are image measures of the flag measures (cf.~the end of Section 3), and the 
$O(d)$ covariance of the  area measures to obtain (with varying constants,  
which may also depend on $K,L$) 
\begin{eqnarray*}
&&\int_{SO(d)}    \iint  \sin^{3-d}\angle(u,v) \, 
\Eam_k(K;d(u,U))\,\Eam_l(\rho L;d(v,V))\,\nu_d(d\rho)\\
&&\qquad\leq\mbox{const}\cdot\int_{SO(d)}\iint {\sin^{3-d}\angle(u,v)}\, S_k(K,du)\, S_l(\rho L ,dv)\,\nu_d(d\rho)\\
&&\qquad\leq\mbox{const}\cdot\iint\int_{SO(d)} {\sin^{3-d}\angle(u,\rho v)}\,\nu_d(d\rho) S_k(K,du)
\, S_l(L ,dv)\\
&&\qquad\leq\mbox{const}\cdot\int_{S^{d-1}}\sin^{3-d}\angle(u,w)\,\mathcal{H}^{d-1}(dw)\cdot S(K)S(L),
\end{eqnarray*}
where $S(K)$ and $S(L)$ denote the surface areas of $K$ and $L$, respectively. Then we  use 
the fact that
$$\int_{S^{d-1}}{\sin^{3-d}\angle(u,w)}\, {\cal H}^{d-1}(dw)=\mathcal{H}^{d-2}(S^{d-2})\cdot \int_0^\pi\sin \alpha \, d\alpha=2\mathcal{H}^{d-2}(S^{d-2})$$
is a finite constant, independent of $u\in S^{d-1}$.

(b) If the support function of $K$ (say) is of class $C^{1,1}$, then all  area measures 
of $K$ are absolutely continuous with bounded density (see \cite[Satz 4.7]{Weil73}), that is
$$
S_k(K,\cdot)\le \mbox{const}\cdot \mathcal{H}^{d-1}\llcorner S^{d-1},
$$
for $k=1,\ldots,d-1$. Hence the estimate
\begin{eqnarray}
&&\iint \sin^{3-d}\angle(u,v)\, \Eam_k(K;d(u,U))\,\Eam_l( L;d(v,V))\nonumber\\
&&\qquad \leq\mbox{const}\cdot\iint\sin^{3-d}\angle(u, v)\, S_k(K,du)\,S_l( L,dv)\label{addref}\\
&&\qquad \leq\mbox{const}\cdot\iint\sin^{3-d}\angle(u,v)\, \mathcal{H}^{d-1}(du)\,S_l( L,dv)\nonumber\\
&&\qquad \leq\mbox{const}\cdot S(L)\nonumber
\end{eqnarray}
shows again that the dominated convergence theorem can be applied.

(c) Let $K,L\subset\rd$ be polytopes in general relative position, and let $k,l\in \{1,\ldots,d-1\}$ 
with $k+l=d$.  Let $F\in\mathcal{F}_k(K)$, $G\in\mathcal{F}_l(L)$ and define $N(F)=L(F)^\perp$, $N(G)=L(G)^\perp$. Since 
$$
(L(F)\cap L(G))^\perp=(L(F)^\perp+L(G)^\perp)^{\perp\perp}=N(F)+ N(G),
$$
we get
\begin{eqnarray}\label{relpos}
\dim(N(F)\cap N(G))&=&\dim(N(F))+\dim(N(G))-(d-\dim(L(F)\cap L(G)))\nonumber\\
&=&\dim(L(F)\cap L(G))=0.
\end{eqnarray}
We  write $N(K,F)$ for the normal cone of $K$ at $F$, $\nu(K,F)=N(K,F)\cap S^{d-1}$, and define 
$N(L,G)$ and $\nu(L,G)$ similarly. Then \eqref{relpos} implies 
$$
\nu(K,F)\cap \pm \nu(L,G)=\emptyset.
$$ 
If $K$ is a polytope, then
$$
S_k(K,\cdot)=\mbox{const}\cdot \sum_{F\in\mathcal{F}_k(K)}\mathcal{H}^k(F)
\mathcal{H}^{k^*}(\nu(K,F)\cap \cdot),
$$
see \cite[(4.2.18)]{S}. 
Therefore, by a compactness argument, we obtain that $\angle(u,v)\in [\delta,\pi-\delta]$, for a constant $\delta\in(0,\pi)$, all 
unit vectors $u$ in the support of $S_k(K,\cdot)$, and all  unit vectors $v$ in the support of $S_l(L,\cdot)$. This implies 
again that the right-hand side in \eqref{addref} is finite, since $|\sin\angle(u,v)|$ is bounded from below by a positive 
constant for all vectors $u,v$ under consideration, and thus 
the dominated convergence theorem can be applied. 

This completes the proof of Theorem \ref{thm3}.

\bigskip

\noindent
{\bf Remarks.} \begin{enumerate}
\item The proof of the preceding theorem shows that \eqref{IR2} holds 
for a pair of convex bodies $K,L\subset\rd$ 
whenever 
$$
\iint \sin^{3-d}\angle(u, v)\, S_k(K,du)\,S_l( L,dv)<\infty.
$$
That this condition is not always satisfied is shown by the example in the next section. 
\item If $K,L\subset \rd$ are polytopes, then equation \eqref{IR2} holds for a particular $k\in \{1,\ldots,d-1\}$, if 
the assumption of general relative position is satisfied just for this particular $k$. 
\item For $U\in G(d,k)$, the $k$-dimensional volume of the orthogonal projection of $K$ onto $U$ 
can be expressed as a special mixed volume, that is
$$
\mathcal{H}^k(K|U)=\frac{\binom{d}{k}}{\mathcal{H}^{d-k}(U_{d-k})}\cdot V(K[k],U_{d-k}[d-k]),
$$
where $U_{d-k}\subset U^\perp$ is any $(d-k)$-dimensional convex body with positive 
$(d-k)$-dimensional volume (this follows from \cite[(5.3.23)]{S} and the linearity properties of mixed volumes). If we choose 
$U_{d-k}=U^\perp\cap B^d$ and recall \eqref{laterefs}, then we obtain 
\begin{align*}
\mathcal{H}^k(K|U)&=\frac{\binom{d}{k}}{\kappa_{d-k}}\cdot V(K[k],U^\perp\cap B^d[d-k])\\
&=\frac{1}{\kappa_{d-k}}\iint F_{k,d-k}(\angle(v,w))\varphi^{k,d-k}(v,V,w,W)\, \\
&\qquad\qquad\times\Omega_{d-k}(U^\perp\cap B^d;d(v,V))\, \Omega_k(K;d(w,W)),
\end{align*}
whenever \eqref{IR2} holds for $K$ and $U^\perp\cap B^d$. This is true, for instance, for $K$ and 
$\nu_{k}$-almost all $U\in G(d,k)$. In this situation, relation \eqref{IR2} also holds if $S_k(K,\cdot)$ is absolutely 
continuous with bounded density with respect to spherical Lebesgue measures or if $K$ 
is a polytope such that $L(F)\cap U^\perp=\{o\}$ for all $k$-dimensional faces $F$ of $K$.  
\end{enumerate}

\section{Example}
Consider the case $d=4$ and $k=l=2$. Hence we have $\bark=\barl=1$. Let $(u,U), (v,V)\in F^\perp(4,1)$ be such that 
$u,v$ are linearly independent and define
\begin{eqnarray*}
\beta&=&\angle(u,v)\in (0,\pi),\\
L&=&u^\perp\cap v^\perp\in G(4,2),\\
\gamma&=&\angle (U|L,V|L)\in[0,\pi],\\
\alpha_U&=&\angle(U,L), \ \alpha_V=\angle(V,L)\in[0,\frac{\pi}2].
\end{eqnarray*}
Here we write $U|L$ for the orthogonal projection of $U$ onto $L$. 
The angle between a one-dimensional linear subspace such as $U$ and a two-dimensional linear 
subspace such as $L$ is defined as the angle between $U$ and $U|L$ 
(unless $U\perp L$ where we
define this angle to be $\pi/2$). 
Note that $\gamma$ is not defined if $U\perp L$ or $V\perp L$, but in these cases, 
the subsequent formulae are  still valid, since then $\cos\alpha_U=0$ or $\cos\alpha_V=0$  
so that the value of $\sin\gamma$ is not relevant. 

The functions $\varphi^{2,2}_{i,j}$, which were introduced in the proof of Proposition \ref{P-dint}, 
can now be expressed in terms of these angles, that is
\begin{eqnarray*}
\varphi^{2,2}_{0,0}(u,U,v,V)&=&\sin^2\beta\cos^2\alpha_U\cos^2\alpha_V\sin^2\gamma,\\
\varphi^{2,2}_{0,1}(u,U,v,V)&=&\sin^2\beta\cos^2\alpha_U
(1-\sin^2\gamma\cos^2\alpha_V),\\
\varphi^{2,2}_{1,0}(u,U,v,V)&=&\sin^2\beta\cos^2\alpha_V
(1-\sin^2\gamma\cos^2\alpha_U),\\
\varphi^{2,2}_{1,1}(u,U,v,V)&=&\sin^2\beta\big(\sin^2\gamma\cos^2\alpha_U\cos^2\alpha_V+\sin^2\alpha_U+\sin^2\alpha_V\big).\\
\end{eqnarray*}

To show this, let us choose unit vectors $u_1,u_2,u_3$ such that $\Lin\{u_1\}=U$ and $u_1,u_2,u_3,u$ is an orthonormal basis of $\mathbb{R}^4$. 
Similarly, we choose unit vectors $v_1,v_2,v_3$ such that $\Lin\{v_1\}=V$ and $v_1,v_2,v_3,v$ is an orthonormal basis of $\mathbb{R}^4$.
From the definition of $\varphi^{2,2}_{0,0}$, provided in the proof of Proposition \ref{P-dint}, we then obtain (omitting 
arguments on the left-hand side, also subsequently)
\begin{align*}
\varphi^{2,2}_{0,0}&=\|u_1\wedge u\wedge v_1\wedge v\|^2=\|u\wedge v\|^2\|u_1|L \wedge v_1|L\|^2\\
&=\sin^2\beta\cdot\| u_1|L\|^2\| v_1|L\|^2\left(\sin\angle( u_1|L, v_1|L)\right)^2\\
&=\sin^2\beta\sin^2\gamma\cos^2\alpha_U\cos^2\alpha_V,
\end{align*}
where $ u_1|L$ denotes the orthogonal projection of $u_1$ onto $L$ (etc.). 
Here we have used the fact that if $\xi$ is a simple unit $k$-vector with associated linear subspace $U$ and $L=U^\perp$, 
and if $\eta_1,\ldots,\eta_l\in\R^d$, then we have
$$
\|\xi\wedge\eta_1\wedge\ldots\wedge\eta_l\|=\|\xi\|\cdot\|\eta_1|L\wedge\ldots\wedge \eta_l|L\|.
$$

Next, we have
\begin{align*}
\varphi^{2,2}_{0,0}+\varphi^{2,2}_{0,1}&=\sum_{j=1}^3\|u_1\wedge v_j\wedge u\wedge v\|^2
=\|u_1\wedge u\wedge v\|^2\sum_{j=1}^3\| v_j|(u_1\wedge u\wedge v)^\perp\|^2\\
&=\|u_1\wedge u\wedge v\|^2=\|u\wedge v\|^2\cdot \|u_1|L\|^2=\sin^2\beta\cos^2\alpha_U,
\end{align*}
where we assume that $u_1,u,v$ are linearly independent, i.e.\ $u_1$ is not orthogonal to $L$. 
In this case, the $3$-vector $u_1\wedge u\wedge v$ is associated with a 3-dimensional linear subspace 
of $\R^4$. Let $a$ be a unit vector orthogonal to this subspace. Then $v_j|(u_1\wedge u\wedge v)^\perp=v_j|\text{Lin}\{a\}$ 
and $a\in v^\perp$. Since 
$v_1,v_2,v_3$ is an orthonormal basis of $v^\perp$, we deduce that
$$
\sum_{j=1}^3\|v_j|(u_1\wedge u\wedge v)^\perp\|^2=\sum_{j=1}^3\langle v_j,a\rangle^2=\|a\|^2=1.
$$
However, even if $u_1,u,v$ are not linearly independent the resulting equation remains true. 

The second assertion now follows from the first one. The third assertion is proved in a similar way. 

The 
remaining assertion is obtained from
\begin{align*}
\varphi^{2,2}_{0,0}+\varphi^{2,2}_{0,1}+\varphi^{2,2}_{1,0}+\varphi^{2,2}_{1,1}&=\sum_{i=1}^3\sum_{j=1}^3
\|u_i\wedge v_j\wedge u\wedge v\|^2
=\sum_{i=1}^3\|u_i\wedge u\wedge v\|^2\\
&=\|u\wedge v\|^2\sum_{i=1}^3\|u_i|L\|^2=2\sin^2\beta.
\end{align*}
Here we used that $u^\perp\cap  v^\perp=L$ is a 2-dimensional linear subspace of $u^\perp$, $u_1,u_2,u_3$ is an 
orthonormal basis of $u^\perp$ and therefore, if  $a,b$ is an orthonormal basis of $L$,  then 
$$
\sum_{i=1}^3\|u_i|L\|^2=\sum_{i=1}^3\left(\langle a,u_i\rangle^2+\langle b,u_i\rangle^2\right)=\|a\|^2+\|b\|^2=2.
$$
Now we compute the coefficients $\alpha_{i,j}$. We have $\gamma(4,2)=\frac 3{2\pi}$ (the constant from \eqref{ECM}) and 
$$c_{1,0}^3=\frac 15,\quad c_{1,1}^3=\frac 1{15}.$$
To verify this, let $a\in S^2$ be arbitrary and  $A=\Lin\{a\}$. Since $\langle A,A\rangle_1=0$, we get
$$
c^3_{1,0}=\int_{G(3,1)}\langle A,V\rangle^2\langle V,A\rangle^2\,\nu_{1}^{3}(dV)=
\frac{1}{\mathcal{H}^2(S^2)}\int_{S^2}\langle v, a\rangle^4\, \mathcal{H}^2(dv)=\frac{1}{5},
$$
where the integral is independent of the particular choice of $a$ by the $O(3)$ invariance of $\nu^3_1$. 
For the calculation of $c_{1,1}^3$, we choose $a=(1,0,0)^\top$, $b=(0,0,1)^\top$, $A=\Lin\{a\}$ and $B=\Lin\{b\}$. 
Then we obtain
$$
c_{1,1}^3=\int_{G(3,1)}\langle A,V\rangle^2\langle V,B\rangle^2\,\nu_{1}^{3}(dV)=
\frac{1}{\mathcal{H}^2(S^2)}\int_{S^2}\langle a, v\rangle ^2\langle v, b\rangle ^2\, \mathcal{H}^2(dv)=\frac{1}{15}.
$$
Finally, from \eqref{combrel} and the subsequent remark we get
\begin{align*}
d_{0,0}^{3,1}&=c_{1,0}^3\cdot 1+c_{1,1}^{3}\cdot 0=\frac{3}{15},\\
d_{0,1}^{3,1}&=c_{1,0}^3\cdot 0+c_{1,1}^{3}\cdot 1=\frac{1}{15},\\
d_{1,0}^{3,1}&=c_{1,0}^3\cdot 0+c_{1,1}^{3}\cdot 2=\frac{2}{15},\\
d_{1,1}^{3,1}&=c_{1,0}^3\cdot 1+c_{1,1}^{3}\cdot 1=\frac{4}{15},\\
\end{align*}
and thus
$$D(3,1)=\frac 1{15}\left(\begin{array}{cc}
3&1\\2&4\end{array}\right),$$
$$D(3,1)\otimes D(3,1)=\frac 1{225}\left(
\begin{array}{cccc}
9&3&3&1\\6&12&2&4\\6&2&12&4\\4&8&8&16
\end{array}\right).$$
Then we deduce
$$\Big(\alpha_{p,q}\Big)_{p,q=0}^1=\pi^2\left( \begin{array}{cc}
16&-4\\-4&1\end{array}\right).$$
The function $\varphi^{2,2}$ is now determined as 
\begin{eqnarray*}
\varphi^{2,2}(u,U,v,V)&=&\pi^2\sin^2\beta\cdot 
(25\sin^2\gamma\cos^2\alpha_U\cos^2\alpha_V\\
&&+\sin^2\alpha_U+\sin^2\alpha_V-4\cos^2\alpha_U-4\cos^2\alpha_V).
\end{eqnarray*}

Let $K$ be a unit square in a 2-dimensional subspace $L\in G(4,2)$. In this case, the normal bundle 
of $K$ has a simple structure and all generalized curvatures of $K$ are either zero or infinite. To be more precise, since $d=4$, for each $(x,u)\in\nor (K)$ there are three generalized principal curvatures which we arrange in increasing order,
$$ 0\le k_1(K;x,u)\le k_2(K;x,u)\le k_3(K;x,u)\le \infty.$$
If $x$ is a relative interior point of $K$, then $(x,u)\in\nor (K)$ if and only if $u\in L^\perp\cap S^3 = S^1_{L^\perp}$, and for such pairs $(x,u)\in \nor (K)$ we have
$$
k_i(K;x,u)= \begin{cases} 0,& i=1,2,\\
\infty,& i=3.
\end{cases}
$$                                
Therefore,
$$
\BK_{\{i\}}(K;x,u)= \begin{cases} 0,& i=1,2,\\ 1,& i=3,
\end{cases}
$$
and $a_3(K;x,u)$ is a unit vector in $L^\perp\cap u^\perp$. If $x\in K$ is not a relative interior point of $K$ 
and $(x,u)\in\nor (K)$, then at least two of the generalized principal curvatures are infinite and 
therefore $\BK_I(K;x,u)=0$, if $|I|=1$. Hence, for $d=4$, $k=2$ (thus $k^* =1$) and $K$ as above, we get

\begin{align}
&\int g(u,U)\,\Omega_2(K;d(u,U))\nonumber\\
&\qquad =\frac 3{2\pi}\int_{\nor (K)} \sum_{i=1}^3 \BK_{\{i\}}(K;x,u) \int_{G^{u^\perp}(3,1)} g(u,V)
\langle V,A_{\{i\}}(K;x,u)\rangle^2\, \nu_1^3(dV)\nonumber\\
&\hspace{6cm}\times {\cal H}^3(d(x,u))\nonumber\\
&\qquad =\frac 3{2\pi}\int_{K\times S^1_{L^\perp}} \int_{G^{u^\perp}(3,1)}g(u,V)
\langle V,L^\perp\cap u^\perp\rangle^2\, \nu_1^3(dV)\,({\cal H}^2\otimes{\cal H}^1)(d(x,u)),\nonumber\\
&\qquad =\frac 3{2\pi}\int_{S^1_{L^\perp}} \int_{G^{u^\perp}(3,1)}g(u,V)
\langle V,L^\perp\cap u^\perp\rangle^2\, \nu_1^3(dV)\,{\cal H}^1(du),\label{neuGl}
\end{align} 
for each bounded measurable function $g$ on $F^\perp(4,1)$.
In particular, the  area measure 
of order $2$ of $K$ is a multiple of ${\cal H}^1$ restricted to the unit circle $S^1_{L^\perp}$ in 
$L^\perp$. 
Clearly, we have
\begin{equation} \label{F22}
F_{2,2}(\beta)\sin^3\beta \to  \mbox{const}\in(0,\infty),\quad \text{ as }\beta\to\pi .
\end{equation}
Subsequently, we show that 
\begin{equation} \label{inf2}
\int\int F_{2,2}(\angle(u,v))\varphi^{2,2}_-(u,U,v,V)\,\Omega_2(K;d(u,U))\, \Omega_2(K;d(v,V))=\infty ,
\end{equation}
where $\varphi^{2,2}_-$ is the negative part of $\varphi^{2,2}$.
Using \eqref{neuGl}, we can write the integral on the left-hand side of \eqref{inf2} as
\begin{eqnarray*}
\frac 9{4\pi^2}\int_{S^1_{L^\perp}}\int_{S^1_{L^\perp}} F_{2,2}(\angle(u,v))
\int_{G^{u^\perp}(3,1)}\int_{G^{v^\perp}(3,1)}
\langle A,U\rangle^2\varphi^{2,2}_-(u,U,v,V)\langle B,V\rangle^2\\ \hspace{6cm}\nu_1^3(dV)\,\nu_1^3(dU)\,
{\cal H}^1(dv)\,{\cal H}^1(du),
\end{eqnarray*}
with $A=u^\perp\cap L^\perp$ and $B=v^\perp\cap L^\perp$. Since $\varphi^{2,2}_{-}\ge 0$, the application of Fubini's theorem 
was justified.  
Choose a fixed direction $\gamma_0\in L$ and parametrize $U,V$ by spherical coordinates as follows: 
$\gamma_U=\angle(\gamma_0,p_LU)$, $\alpha_U=\angle(U,L)$, $\gamma_V=\angle(\gamma_0,p_LV)$, 
$\alpha_V=\angle(V,L)$. Denoting $\gamma=\gamma_U-\gamma_V$, we get that the negative part 
$\varphi^{2,2}_-$ of $\varphi^{2,2}$ is bounded from below by 
$$\pi^2\sin^2\beta(3-25\sin^2\gamma),$$
provided that $\alpha_U,\alpha_V\leq\pi/4$. In fact, under this restriction we have
\begin{align*}
\varphi^{2,2}_+-\varphi^{2,2}_-&=\varphi^{2,2}\le \pi^2\sin^2\beta\left({25}\sin^2\gamma+\frac{1}{2}+\frac{1}{2}-4\frac{1}{2}-4\frac{1}{2}\right)\\
&\le\pi^2\sin^2\beta\left({25}\sin^2\gamma-3\right),
\end{align*}
from which the estimate for $\varphi^{2,2}_-$ follows.

Thus, for linearly independent $u,v\in S^1_{L^\perp}$, on the set
$$D_{(u,v)}=\left\{(U,V)\in  G^{u^\perp}(3,1)\times G^{v^\perp}(3,1): 0<\alpha_U,\alpha_V<\frac{\pi}4,\,|\sin(\gamma_U-\gamma_V)|\leq \frac 15\right\}$$
we have
$$\varphi^{2,2}_-(u,U,v,V)\geq 2\pi^2\sin^2\beta.$$
The integration over the Grassmannians $G^{u^\perp}(3,1)$ and $G^{v^\perp}(3,1)$ can be written in spherical coordinates ($0<\gamma_U,\gamma_V<2\pi$, $0<\alpha_U,\alpha_V<\pi/2$) as
$$\nu_1^3(dU)=\frac 1{2\pi}\, d\gamma_U\, \cos\alpha_U \,d\alpha_U,\quad 
\nu_1^3(dV)=\frac 1{2\pi}\, d\gamma_V\, \cos\alpha_V \,d\alpha_V.$$
Since $\langle A,U\rangle^2=\sin^2\alpha_U$ and $\langle B,V\rangle^2=\sin^2\alpha_V$, we obtain for all linearly independent vectors  $u,v\in S^1_{L^\perp}$ 
\begin{eqnarray*}
\lefteqn{\int\int_{D_{(u,v)}} \langle A,U\rangle^2\langle B,V\rangle^2\, \nu_1^3(dU)\,\nu_1^3(dV)}\\
&=&\frac 1{(2\pi)^2} \left( \int_0^{2\pi}\int_0^{2\pi} {\bf 1}\{|\sin(\gamma_U-\gamma_V)|\leq\frac 15\}\, d\gamma_U\, d\gamma_V\right)  \left(\int_0^{\pi/4} \sin^2\alpha\cos\alpha\, d\alpha\right)^2\\
&=&\frac 1{4\pi^2} \left(2\pi\cdot 4\arcsin\frac 1{5}\right)\left(\frac{\sqrt{2}}{12}\right)^2\\
&=&\frac 1{36\pi}\arcsin\frac 1{5}.
\end{eqnarray*}
Consequently,
\begin{eqnarray*}
\lefteqn{\int_{G^{u^\perp}(3,1)}\int_{G^{v^\perp}(3,1)} \langle A,U\rangle^2\varphi^{2,2}_-(u,U,v,V)\langle B,V\rangle^2\,\nu_1^3(dV)\,\nu_1^3(dU)}\hspace{6cm}\\
&\geq& \left(\frac {\pi}{18}\arcsin\frac 1{5}\right)\sin^2\beta.
\end{eqnarray*}
Taking into account \eqref{F22}, it is clear that \eqref{inf2} holds.

\end{document}